\documentclass{birkjour}

 \newtheorem{thm}{Theorem}[section]
 \newtheorem{cor}[thm]{Corollary}
 \newtheorem{lem}[thm]{Lemma}
 \newtheorem{prop}[thm]{Proposition}
 \theoremstyle{definition}
 \newtheorem{defn}[thm]{Definition}
 \newtheorem{ex}[thm]{Example}
 \theoremstyle{remark}
 \newtheorem{rem}[thm]{Remark}

 \numberwithin{equation}{section}

\newcommand{\range}[1]{\mathsf{R}\left( #1 \right)}

\newcommand{\kernel}[1]{\mathsf{N}\left( #1 \right)}
\newcommand{\finspan}[1]{\mathsf{span}\left( #1 \right)}

\newcommand{\norm}[2]{
\left\| #2 \right\|_{#1}
}
\def\Hil{\mathcal{H}}

\newcommand{\BL}[1]{
{\mathcal B} \left( #1 \right)
}

\begin{document}

\title[$U$-cross Gram matrices and their invertibility]
 {$U$-cross Gram matrices and their invertibility}

\author[P. Balazs]{{Peter Balazs}}
\address{ Acoustics Research Institute
 Austrian Academy of Sciences, Vienna, AUSTRIA.}
\email{peter.balazs@oeaw.ac.at}

\author[M. Shamsabadi]{Mitra Shamsabadi}
\address{Department of Mathematics and Computer Sciences, Hakim Sabzevari University, Sabzevar, IRAN.}
\email{mi.shamsabadi@hsu.ac.ir}

\author[A. Arefijamaal]{Ali Akbar Arefijamaal}
\address{Department of Mathematics and Computer Sciences, Hakim Sabzevari University, Sabzevar, IRAN.}
\email{arefijamaal@hsu.ac.ir;arefijamaal@gmail.com}

\author[A. Rahimi]{Asghar Rahimi}
\address{ Department of Mathematics, University of Maragheh,
Maragheh, IRAN.}
\email{rahimi@maragheh.ac.ir}


\begin{abstract}
The Gram matrix is defined for Bessel sequences by combining
synthesis  with subsequent analysis operators. If different
sequences are used and an operator $U$ is inserted we reach so called
$U$-cross Gram matrices. This can be seen as reinterpretation of
the matrix representation of operators using frames. In this paper
we investigate some necessary or sufficient conditions for
Schatten $p$-class properties and the invertibility of $U$-cross
Gram matrices. In particular, we show that under mild conditions
the pseudo-inverse of a $U$-cross Gram matrix can always be
represented as a $U$-cross Gram matrix with dual frames of the given
ones. We link some properties of $U$-cross Gram matrices to approximate duals.
Finally, we state several stability results. More precisely, it is
shown that the invertibility of $U$-cross Gram matrices is preserved
under small perturbations.
\end{abstract}
\maketitle

\section{Introduction and motivation}
Some operator equations, e.g. in acoustics \cite{Kreuzeretal09} and vibration simulation \cite{kreizxxl1} cannot be treated analytically, but have to be solved numerically. Depending on the problem this can be done using a boundary element method \cite{sauter2010boundary} or finite element method \cite{brennscott1} approach. Thereby
operator equations $O f = b$, are transferred to matrix levels to be able to be treated numerically \cite{sauter2010boundary}.
A standard approach for that, the Galerkin method \cite{HarSch06}, is using orthonormal basis (ONB) $\{e_i\}_{i\in I}$ and investigate the matrix
$M_{k,l} :=\left\langle O e_l, e_k\right\rangle$ \cite{HarSch06} solving $M c = d$  for $d = \{d_l\}_{l\in I}= \left\{\left< b , e_l \right>\right\}_{l\in I}$. 
More recently frames are used for such a discretization \cite{xxlhar18,Stevenson03}. On a more theoretical level, it is well known that operators can be represented by matrices using orthonormal bases \cite{gohberg1}.  Recently,  the theory for frames has been settled for this theoretical approach \cite{xxlframoper1,xxlhar18,xxlriek11}.

Those matrices are constructed by concatenating the given operator $U$ with the synthesis and the analysis operators. Therefore they can be considered as generalizations of Gram matrices.
In this article we study those so called $U$-cross Gram matrices and investigate their invertibility, in particular.
The composition and the invertibility of $U$-cross Gram
matrices are our main questions in this paper. In addition, it
is very natural to ask whether the composition and more intricate and interesting, the inverses of $U$-cross
Gram matrices can be stated as  $U$-cross Gram matrices.
The affirmative answer to these questions will be useful in applied frame theory, mentioned above.
Similar questions are studied  for
frame multipliers, $K$-frame multipliers and fusion frame multipliers in \cite{xxlmult1,mitrak,mitra,uncconv2011, balsto09new}  and matrix representations  \cite{xxlriek11,xxlrieck11,harbr08}.

This paper is built up as follows: In Section \ref{sec:notation} we fix the notation and collect results needed.
In Section \ref{sec:basicdef} we give the basic definition of $U$-cross Gram matrices, some examples, look at Schatten $p$-class properties and investigate this concept for Riesz sequences.
In Section \ref{sec:pivv}  we look at the pseudo-inverses of $U$-cross Gram matrices. In particular, we show under which circumstances this can be written as such a matrix again.
In Section \ref{sec.apprdual} we look at sufficient and necessary conditions on the $U$-cross Gram matrix to imply the involved sequences to be approximate duals. And finally
in Section \ref{sec:stability} we investigate how stable the invertibility of this matrix is regarding the perturbation of the operator or the sequences.

\section{Notation} \label{sec:notation} Throughout this paper, $\mathcal{H}$ is a separable Hilbert space,
$I$ a countable index set and $I_{\mathcal{H}}$ the identity
operator on $\mathcal{H}$. The orthogonal projection on a subspace
$V\subseteq\mathcal{H}$ is denoted by $\pi_V$. We will denote the
set of all linear and bounded operators between Hilbert spaces
$\mathcal{H}_{1}$ and $\mathcal{H}_{2}$ by
$\BL{\mathcal{H}_{1},\mathcal{H}_{2}}$ and for
$\mathcal{H}_{1}=\mathcal{H}_{2}=\mathcal{H}$, it is represented
by $\BL{\mathcal{H}}$.
 We denote the range and the null spaces
of an operator $U \in \BL{\mathcal{H}_{1},\mathcal{H}_{2}}$ by $\range{U}$ and $\kernel{U},$ respectively.  
For a closed range operator $U \in \BL{\mathcal{H}_{1},\mathcal{H}_{2}}$, the pseudo-inverse of $U$ is the unique
operator $U^{\dag} \in \BL{\mathcal{H}_{2},\mathcal{H}_{1}}$ satisfying that
\begin{eqnarray*}
\kernel{U^{\dag}}=\range{U}^{\perp},\
\range{U^{\dag}}=\kernel{U}^{\perp},\
\text{and}\
UU^{\dag}U=U.
\end{eqnarray*}
If $U$ has closed range, then $U^*$ has
closed range and
$\left(U^*\right)^{\dag}=\left(U^{\dag}\right)^{*}$, see e.g.
\cite[Lemma 2.5.2]{Chr08}.

A sequence $\Phi=\{\phi_{i}\}_{i\in I}$ in a separable Hilbert space
$\mathcal{H}$ is a \textit{frame} if there exist constants $A_{\Phi},B_{\Phi}>0$
such that for all $f \in \Hil$ 
\begin{eqnarray} \label{sec:frame1}
A_{\Phi}\left\|f\right\|^{2}\leq \sum_{i\in I}\left|\left\langle f,\phi_{i}\right\rangle\right|^{2}\leq B_{\Phi}\left\|f\right\|^{2}.
\end{eqnarray}
The numbers $A_{\Phi}$ and $B_{\Phi}$ are called the \textit{frame bounds}. If $\{\phi_{i}\}_{i\in I}$ is assumed to satisfy the right hand of \eqref{sec:frame1}, then it is called a \textit{Bessel sequence} with Bessel bound $B_{\Phi}$. We say that a sequence $\{\phi_{i}\}_{i\in I}$ in $\mathcal{H}$ a \textit{frame sequence} if it is a frame for $\overline{\textrm{span}}\{\phi_{i}\}_{i\in I}$.
For a Bessel sequence $\Phi=\{\phi_{i}\}_{i\in I}$, the \textit{synthesis operator} $T_{\Phi}:\ell^{2}\rightarrow \mathcal{H}$ is defined by
\begin{eqnarray*}
T_{\Phi}\{c_{i}\}_{i\in I}=\sum_{i\in I}c_{i}\phi_{i}.
\end{eqnarray*}
 Its adjoint operator $T_{\Phi}^{*}:\mathcal{H}\rightarrow \ell^{2}$; the so called  \textit{analysis operator}  is given by
 \begin{eqnarray*}
 T_{\Phi}^{*}f=\left\{\left\langle f,\phi_{i}\right\rangle\right\}_{i\in I},\qquad.
 \end{eqnarray*}
The operator $S_{\Phi}:\mathcal{H}\rightarrow \mathcal{H}$,
which is defined by $S_{\Phi}f=T_{\Phi}T_{\Phi}^{*}f=\sum_{i\in I}\left\langle f,\phi_{i}\right\rangle \phi_{i}$, for all $f\in \mathcal{H}$, is called the \textit{frame operator}.
For a frame ${\Phi}$ the operator {$T_{\Phi}$}  is onto, {$T^*_{\Phi}$} is one-to-one, and ${S_{\Phi}}$ is positive, self-adjoint and invertible \cite{Chr08}. Also, if $B_{\Phi}$ is the Bessel bound of $\Phi$, then
 \begin{eqnarray*}
\left\|T_{\Phi}c\right\|\leq \sqrt{B_{\Phi}} \|c\|, 
    \end{eqnarray*}
   for every sequence of scalars $c=\{c_{i}\}_{i\in I} \in \ell^2$. Note that those operators can be defined for any sequence \cite{xxlstoeant11} resulting in potential  unbounded operators.
    We call a complete Bessel sequence an \textit{upper semi-frame} \cite{jpaxxl09,antbal12}.

A \textit{dual} for a Bessel sequence $\Phi=\{\phi_{i}\}_{i\in I}\subseteq \mathcal{H}$ is a Bessel sequence $\Psi=\{\psi_{i}\}_{i\in I}$ in $\mathcal{H}$ such that
\begin{eqnarray*}\label{dual}
f=\sum_{i\in I}\left\langle f,\psi_{i}\right\rangle \phi_{i},\qquad(f\in \mathcal{H}). 
\end{eqnarray*}
For a frame $\Phi$ it is obvious to see that the Bessel sequence $\left\{S_{\Phi}^{-1}\phi_{i}\right\}_{i\in I}$ is a dual and is itself a frame again. This dual, denoted by $\widetilde{\Phi}=\left\{\widetilde{\phi_{i}}\right\}_{i\in I}$, is called the \textit{canonical dual}.
Note that this is the only equivalent dual, i.e $\range{T^*_{\Phi}}=\range{T^*_{\widetilde{\Phi}}}$.

Recall that  Bessel sequences $\Phi$ and  $\Psi$  in $\mathcal{H}$ are called  \textit{approximate dual frames}, if
\begin{eqnarray*}
\left\|T_{\Phi}T_{\Psi}^*-I_{\mathcal{H}}\right\|< 1\quad \textrm{or}\quad \left\|T_{\Psi}T_{\Phi}^*-I_{\mathcal{H}}\right\|< 1.
\end{eqnarray*}
Note that if $\Phi$ and $\Psi$ are approximately dual frames, then the operator $T_{\Psi}T_{\Phi}^{*}$ is invertible, in other words $\Phi$ and $\Psi$ are a reproducing pair \cite{spexxl14} or pseudo-dual \cite{liog04}.
Hence each $f\in \mathcal{H}$ has the representation
\begin{eqnarray*}
f=(T_{\Psi}T_{\Phi}^{*})^{-1}T_{\Psi}T_{\Phi}^{*}f=\sum_{i\in I}\left\langle f,\phi_{i}\right\rangle (T_{\Psi}T_{\Phi}^{*})^{-1}\psi_{i}.
\end{eqnarray*}
In particular, $\Phi$ and $(T_{\Psi}T_{\Phi}^{*})^{-1}\Psi$ are a pair of dual frames \cite{appchr}.

A \textit{Riesz basis} for $\mathcal{H}$ is a family of the form $\{Ue_{i}\}_{i\in I}$, where $\{e_{i}\}_{i\in I}$ is an orthonormal basis for $\mathcal{H}$ and $U:\mathcal{H}\rightarrow \mathcal{H}$ is a bounded bijective operator. Every Riesz basis is a frame and has a biorthogonal sequence which is also its unique dual  \cite{Chr08}.
The following proposition will be used in this manuscript.
\begin{prop}\label{Eq.cond.Riesz}\cite{xxlstoeant11,Chr08}
For a sequence $\Phi= \{\phi_{i}\}_{i\in I}$ in $\mathcal{H}$, the following conditions are equivalent:
\begin{enumerate}
\item\label{Re1} $\Phi$ is a Riesz basis for $\mathcal{H}$.
    \item\label{Re2} $\Phi$ is complete in $\mathcal{H}$ and there exist constants $A,B>0$ such that
        \begin{eqnarray}\label{RS}
        A\sum_{i\in I}\left|c_{i}\right|^{2}\leq \left\|\sum_{i\in I}c_{i}\phi_{i}\right\|^{2}\leq B\sum_{i\in I}\left|c_{i}\right|^{2},
        \end{eqnarray}
                for every finite scalar sequence $\{c_{i}\}_{i\in I}$.
                \item\label{Re3}  ${\Phi}$ is a frame and $T_{\Phi}$ is one to one.
        \item\label{Re4} $T_{\Phi}^{*}$ is onto  and ${\Phi}$ is an upper semi frame.
\end{enumerate}
\end{prop}
A sequence $\{\phi_{i}\}_{i\in I}$ satisfying (\ref{RS}) for all finite sequences $\{c_{i}\}_{i\in I}$ is called a \textit{Riesz sequence}.
Therefore a Riesz basis is a complete Riesz sequence.

For more details of frame theory see \cite{framepsycho16,Casaz1,Chr08}.\\

Recall that if $U$ is a compact operator on a separable Hilbert
space $\mathcal{H}$, then there exist orthonormal sets $\{e_{n}\}_{n\in I}$
and $\{\sigma_{n}\}_{n\in I}$ in $\mathcal{H}$ such that
$$Ux=\sum_{n\in I}\lambda_{n}\langle x, e_{n}\rangle\sigma_{n},  $$ for
$x\in \mathcal{H}$, with $\lambda_{n} \in c_0$, i.e. $\lim \limits_{n \rightarrow \infty} \lambda_{n} = 0$. $\lambda_{n}$ is called the $n$th singular
value of $U$. Given $0<p<\infty$, we define the Schatten $p$-class
of $\mathcal{H}$, denoted $S_{p}(\mathcal{H})$, as the space of all compact
operators $U$ on $\mathcal{H}$ for which singular value sequence
$\{\lambda_{n}\}_{n\in I}$ belongs to $\ell^{p}$. In this case, $S_{p}(\mathcal{H})$
is a Banach space with the norm
\begin{eqnarray}\label{norm}
\|U\|_{p}=\left(\sum_{n\in I}|\lambda_{n}|^{p}\right)^{\frac{1}{p}}.
\end{eqnarray}
The Banach space $S_{1}(\mathcal{H})$ is called the\textit{ trace class} of $\mathcal{H}$ and
$S_{2}(\mathcal{H})$ is called the \textit{Hilbert-Schmidt class.}

We know that $U\in S_p(\mathcal{H})$ if and only if
$\{\|{Ue_n}\|\}_{n\in I} {\in \ell^p}$, for all orthonormal bases
$\{e_n\}_{n\in I}$. For $0<p \leq 2$ it is even enough to have the
property for a single orthonormal basis, i.e. $U \in S_p
(\mathcal{H})$ if and only if $\{\|{Ue_n}\|\}_{n\in I}{\in
\ell^p}$, for some orthonormal basis $\{e_n\}_{n\in I}$. It is
proved that $S_p(\mathcal{H})$ is a two sided $*$-ideal of
$\BL{\mathcal{H}}$, that is, a Banach algebra under the norm
(\ref{norm}) and the finite rank operators are dense in
$(S_p(\mathcal{H}), \|.\|_p)$.
This can be extended to operators between separate spaces;
according to Theorem 7.8(c) \cite{weidm80} if $U_1\in
\BL{\mathcal{H}_1,\mathcal{H}_2}$, then $\|U_1 U_2\|_p
\leq\|U_1\|\|U_2\|_p$ and $\|U_2 U_1\|_p\leq\|U_1\|\|U_2\|_p$ for
all $U_2\in S_p(\mathcal{H})$. For more information about these
operators, see \cite{Pi,Schatten,weidm80,zhuo}.

In the following theorem, the trace norm of bounded operators is
computed by orthonormal bases.

\begin{thm}\label{ring}\cite{weidm80}
Let $U\in \BL{\mathcal{H}_{1},\mathcal{H}_{2}}$. Then  $U\in
S_p(\mathcal{H}_1,\mathcal{H}_2)$ if and only if
\begin{eqnarray*}
\left\|U\right\|_{p}=\sup \left( \sum \limits_{i\in I}\left|\left\langle Ue_{i},f_{i}\right\rangle\right|^p \right)^{1/p}  < \infty , 
\end{eqnarray*}
where the supremum is taken over all orthonormal bases $\{e_{i}\}_{i\in I}$ of $\mathcal{H}_{1}$ and $\{f_{i}\}_{i\in I}$ of $\mathcal{H}_{2}$.
\end{thm}

 Finally, recall \cite{gr01} that for every matrix operator $M=(M_{k,l})$ on $\ell^2$ we have the mixed norm
 \begin{eqnarray*}
 \| M \|_{p,q} := \left(  \sum \limits_{k\in I} \left( \sum \limits_{l\in I}  \left| M_{k,l} \right|^{q} \right)^{p/q}\right)^{1/p}.
 \end{eqnarray*}
 It is called the Frobenius norm when $p=q=2.$

  We will use the following criterion for the invertibility
of operators.
\begin{prop} \cite{gohberg1}\label{invOP}
Let $U_1:\mathcal{H}\rightarrow \mathcal{H}$ be bounded and invertible on $\mathcal{H}$.
Suppose that $U_2:\mathcal{H}\rightarrow \mathcal{H}$ is a bounded
 operator and $\|U_{2}h-U_{1}h\|\leq \upsilon\|h\|$ for all $h\in \mathcal{H}$,
  where $\upsilon\in[0,\frac{1}{\|U_1^{-1}\|}
)$. Then
 $U_2$ is invertible on $\mathcal{H}$ and
    $U_2^{-1}=\sum_{k=0}^{\infty}[U_1^{-1}(U_1-U_2)]^{k}(U_1)^{-1}$.

\end{prop}

\section{$U$-cross Gram matrices \label{sec:basicdef}}
In this section, we define $U$-cross Gram matrices and introduce their properties.
\begin{defn}
Let $\Psi=\{\psi_{i}\}_{i\in I}$  and $\Phi=\{\phi_{i}\}_{i\in I}$
be  Bessel sequences in  Hilbert spaces $\mathcal{H}_{1}$ and
$\mathcal{ H}_{2}$, respectively. For $U\in
\BL{\mathcal{H}_{1},\mathcal{H}_{2}}$, the matrix
$\mathbf{G}_{U,\Phi,\Psi}$ given by
\begin{eqnarray}\label{gram}
\left(\mathbf{G}_{U,\Phi,\Psi}\right)_{i,j}=\left\langle U\psi_{j},\phi_{i}\right\rangle,\qquad\left(i,j\in I\right),
\end{eqnarray}
 is  called the \textit{$U$-cross Gram matrix}. If $\mathcal{H}_{1}=\mathcal{H}_{2}$ and $U=I_{\mathcal{H}_{1}}$ it is called the \textit{cross Gram matrix} and denoted by $\mathbf{G}_{\Phi,\Psi}$. We use $\mathbf{G}_{\Phi}$ for $\mathbf{G}_{\Phi,\Phi}$; the so called \textit{Gram matrix} \cite{Chr08}.
\end{defn}
 Note this is just another viewpoint to the matrix
representation of operators \cite{xxlframoper1}. 
In the next
lemma, we rephrase needed results in that paper for the $U$-cross
Gram matrices viewpoint.
\begin{lem}
Let $\Phi=\{\phi_{i}\}_{i\in I}$ and $\Psi=\{\psi_{i}\}_{i\in I}$
be two Bessel sequences in $\mathcal{H}_{2}$ and
 $\mathcal{H}_{1}$. Also, let $U\in
\BL{\mathcal{H}_{1},\mathcal{H}_{2}}$. The following assertions
hold.
 \begin{enumerate}
\item[(1)]  $\mathbf{G}_{U,\Phi,\Psi}=T_{\Phi}^{*}UT_{\Psi}$. In particular, the $U$-cross Gram matrix $\mathbf{G}_{U,\Phi,\Psi}$ defines a bounded operator on  $\ell^{2}$  and
$
\left\|\mathbf{G}_{U,\Phi,\Psi}\right\|\leq \sqrt{B_{\Phi}B_{\Psi}}\|U\|.
$
\item[(2)] $\left(\mathbf{G}_{U,\Phi,\Psi}\right)^{*}=\mathbf{G}_{U^{*},\Psi,\Phi}$.
\end{enumerate}
\end{lem}
\begin{proof} (1) is shown in \cite{xxlframoper1}, (2) is trivial.
\end{proof}
As in \cite{xxlframoper1} we have the representation of an
operator $U$ by
\begin{equation} \label{eq:oprec}
U=T_{ \Phi^{d}} \mathbf{G}_{U,\Phi,\Psi}T_{ \Psi^d}^*, 
\end{equation}
where $\Phi$ and $\Psi$ are frames with dual frames $\Phi^d$ and $\Psi^d$, respectively.

For any sequence $\Phi$ in $\mathcal{H},\mathbf{G}_{\Phi}$ is a
bounded operator in $\ell^{2}$ if and only if $\Phi$ is a Bessel
sequence \cite{Chr08}. This result, naturally, does not hold for
$\mathbf{G}_{\Phi,\Psi}$ and $\mathbf{G}_{U,\Phi,\Phi}$. For
example, if $\Phi=\left\{ke_{k}\right\}_{k\in I}$,
$\Psi=\left\{\frac{1}{k}e_{k}\right\}_{k\in I}$ and
$U=I_{\mathcal{H}}$, where $\{e_{k}\}_{k\in I}$ is an orthonormal
basis of $\mathcal{H}$, then it is easy to see that
$\mathbf{G}_{U,\Phi,\Psi}$ is a bounded operator in $\ell^{2}$,
however, $\Psi$ is not  Bessel sequence. Moreover, if $U\in
\BL{\mathcal{H}}$ is defined as $Ue_{k}=\frac{1}{k^2}e_{k}$, $k\in
\mathbb{N}$, then $U\Phi=\Psi$, and therefore,
\begin{eqnarray*}
\mathbf{G}_{U,\Phi,\Phi}=T_{\Phi}^*UT_{\Phi}=T_{\Phi}^*T_{\Psi}=I_{\ell^{2}}
\end{eqnarray*}
is bounded, even invertible, but $\Phi$ is not Bessel sequence (For similar examples see \cite{bstable09}).
To find those sequences for which $\mathbf{G}_{\Phi,\Psi}$ (or even $\mathbf{G}_{U,\Phi,\Psi}$) is invertible, is connected to the concepts of reproducing pairs \cite{spexxl14} and pseudo frames \cite{liog04}. 

By these examples we see that even for nice operator $U$ we cannot deduce properties of the sequence $\Phi$ and $\Psi$.
We will investigate the converse in this paper, which properties of $U$ can be deduced from those of $\mathbf{G}_{U,\Phi,\Psi}$ for nice sequences $\Psi$ and $\Phi$.

\begin{rem} Let $\Phi$, $\Psi$, $\Theta$ and $\Xi$ be Bessel sequences in $\mathcal{H}$. Let $U_{1}$  and $U_{2}\in \BL{\mathcal{H}}$. Then
\begin{enumerate}
\item[(1)]
$
   \mathbf{G}_{U_{1},\Phi,\Psi}\mathbf{G}_{U_{2},\Theta,\Xi}=T_{\Phi}^{*}
   U_{1}T_{\Psi}T_{\Theta}^{*}U_{2}T_{\Xi}
 = \mathbf{G}_{\left(U_{1}T_{\Psi}T_{\Theta}^{*}U_{2}\right),\Phi,\Xi}.
   $
\item[(2)]
    $\mathbf{G}_{U_{1},\Phi,\Psi}\mathbf{G}_{U_{2},\Psi,\Xi}=T_{\Phi}^{*}
   U_{1}T_{\Psi}T_{\Psi}^{*} U_{2}T_{\Xi}=T_{\Phi}^{*}
   U_{1}S_{\Psi}U_{2}T_{\Xi}=
    \mathbf{G}_{\left( U_{1}S_{\Psi}U_{2} \right),\Phi,\Xi}$.
\end{enumerate}
Suppose $\Psi$ is a frame, $\Psi^\dagger$ any dual and $\widetilde \Psi$ the canonical dual  of $\Psi$. Let $\Delta=\{\delta_{i}\}_{i\in
I}$  be the standard orthonormal basis of $\ell^{2}$, then we obtain
\cite{xxlframoper1}
\begin{enumerate}
\item[(3)]
    $\mathbf{G}_{U_{1},\Phi,\Psi}\mathbf{G}_{U_{2},\Psi^{\dag},\Xi}=\mathbf{G}_{U_{1}
    ,\Phi,\Psi^{\dag}}\mathbf{G}_{U_{2},\Psi,\Xi}=
    \mathbf{G}_{\left(U_{1} U_{2}\right),\Phi,\Xi}$.
\setcounter{enumi}{3}
\item[(4)] $
    \mathbf{G}_{S_{\Psi},\Psi,\widetilde{\Psi}}=\mathbf{G}_{S_{\Psi},
    \widetilde{\Psi},\Psi}=\mathbf{G}_{\Psi}$.
 \item[(5)]
     $\mathbf{G}_{S_{\Psi}^{-1},\Psi,\widetilde{\Psi}}=
     \mathbf{G}_{S_{\Psi}^{-1},\widetilde{\Psi},\Psi}=
     \mathbf{G}_{\widetilde{\Psi}}$.
      \item[(6)]
          $\mathbf{G}_{T_{\Phi}^{*},\Delta,\Psi}=\mathbf{G}_{\Phi,\Psi}$.
          In fact
\begin{eqnarray*}
      \left(\mathbf{G}_{T_{\Phi}^{*},\Delta,\Psi}\right)_{i,j}&=&
\left\langle T_{\Phi}^{*}\psi_{j},\delta_{i}\right\rangle\\
&=&\sum_{k\in I}\langle \psi_{j},\phi_{k}\rangle\langle \delta_{k},
\delta_{i}\rangle\\
&=&\langle\psi_{j},\phi_{i}\rangle=
      (\mathbf{G}_{\Phi,\Psi})_{i,j}. 
\end{eqnarray*}
\end{enumerate}
Let $\Psi=\{\psi\}_{i\in I}$ be a Riesz basis in $\mathcal{H}$
 then we have 
\begin{enumerate}
\setcounter{enumi}{6} 
\item[(7)]\label{Re8} $
\mathbf{G}_{T_{\Psi}^{*},\Delta,\widetilde{\Psi}}=\mathbf{G}_{S_{\Psi},\widetilde{\Psi},\widetilde{\Psi}}=
\mathbf{G}_{S_{\Psi}^{-1},\Psi,\Psi}=I.
$  More precisely, by using (\ref{gram}) and the biorthogonality of a Riesz basis
and its canonical dual \cite[Theorem 5.5.4]{Chr08}, we
obtain
\begin{eqnarray*}
\left(\mathbf{G}_{T_{\Psi}^{*},\Delta,\widetilde{\Psi}}\right)_{i,j}&=&
\left\langle T_{\Psi}^{*}\widetilde{\psi}_{j},\delta_{i}\right\rangle\\
&=&\left\langle\widetilde{\psi}_{j},\psi_{i}\right\rangle=\delta_{i,j}.
\end{eqnarray*}
The proof of the other statements are obvious by the biorthogonal property.
\end{enumerate}
 \end{rem}

\subsection{Schatten $p$-classes}
An operator $O$ is compact if and only if \cite{weidm80} $\lim \limits_{k \rightarrow \infty} \| O e_k \| = 0$, for all ONBs $\{e_k\}_{k\in I}$. This is true if and only if $\lim \limits_{k \rightarrow \infty} \sum \limits_{l\in I} \left| \left< O e_k, f_l \right> \right|^2 = 0$, for all orthonormal
 bases $\{ e_n \}_{n\in I}$ and $\{ f_n \}_{n\in I}$.
So, using the canonical  basis of $\ell^2$ for our setting, this
  means that if $\mathbf{G}_{U,\Phi,\Psi}$ is compact, then $\lim \limits_{i\rightarrow \infty} \sum \limits_{l\in I}  \left|
\left\langle U\psi_{i},\phi_{l}\right\rangle \right|^2=0$. As $O$
is compact, if only $O^*$ is compact, this is also equivalent to $
\lim \limits_{i\rightarrow \infty} \sum \limits_{l\in I}  \left|
\left\langle U^*\phi_{i},\psi_{l}\right\rangle \right|^2=0 $.
In particular, this implies that $\lim \limits_{i\rightarrow
\infty}  \left\langle U\psi_{i},\phi_{i}\right\rangle =0$.
Naturally, Frobenius matrices correspond to Hilbert-Schmidt
operator \cite{xxlphd1}. Therefore, if $\sum_{i\in I}\sum_{j\in
I}\left|\left\langle
U\psi_{i},\phi_{j}\right\rangle\right|^{2}<\infty$, then
$\mathbf{G}_{U,\Phi,\Psi}$ is Hilbert-Schmidt, and therefore
compact. More generally, this is true if $\|
\mathbf{G}_{U,\Phi,\Psi}\|_{p,2} < \infty$, for $1\le p < \infty$. 
This allows to formulate the following results for Bessel
sequences:
\begin{cor} \label{sec:schattenmat1}
Let $U\in \BL{\mathcal{H}}$, $\Phi=\{\phi_{i}\}_{i\in I}$ and $\Psi=\{\psi_{i}\}_{i\in I}$
be Bessel sequences in $\mathcal{H}$.
Then the following assertions hold.
\begin{enumerate}
\item[(1)]  If the operator $U$ is compact, the matrix $\mathbf{G}_{U,\Phi,\Psi}$ is also compact. In particular, $\lim \limits_{i\rightarrow \infty} \sum \limits_{l\in I}  | \left< U \psi_{i},\phi_{l} \right> |^2=0$.
\item[(2)]  If the operator $U$ is Schatten $p$-class, the matrix $\mathbf{G}_{U,\Phi,\Psi}$ is Schatten $p$-class. In this case $\left(\sum \limits_{i\in I} \left| \left< U \psi_{i},\phi_{i} \right> \right|^p \right)^{1/p}<\infty$ and $\left\|\left< U \psi_{i},\phi_{l} \right> \right\|_{p,2} < \infty$. 
In particular:
\begin{enumerate}
\item[(2a.)] If the operator $U$ is trace-class, then $\mathbf{G}_{U,\Phi,\Psi}$ is trace-class, if and only if  $\sum \limits_{i\in I} \left| \left< U \psi_{i},\phi_{i} \right> \right| < \infty$. 
\item[(2b.)] If the operator $U$ is Hilbert-Schmidt, then $\mathbf{G}_{U,\Phi,\Psi}$ is Hilbert-Schmidt, if and only if $\sum \limits_{i\in I}\sum \limits_{l\in I}\left|\left\langle U\psi_{i},\phi_{l}\right\rangle\right|^{2}<\infty$.
\end{enumerate}
 \end{enumerate}
\end{cor}
\begin{proof}
This follows from the ideal property of the considered operator spaces, as $\mathbf{G}_{U,\Phi,\Psi}=T_{\Phi}^{*}UT_{\Psi}$, as well as the above comments.
\end{proof}

For frames we can show equivalent conditions:
\begin{lem}
Let $U\in \BL{\mathcal{H}}$, $\Phi=\{\phi_{i}\}_{i\in I}$ and $\Psi=\{\psi_{i}\}_{i\in I}$
be frames in $\mathcal{H}$.
Then the following assertions hold.

\begin{enumerate}
\item[(1)]  The operator $U$ is compact, if and only if $\mathbf{G}_{U,\Phi,\Psi}$ is compact. In this case $\lim \limits_{i\rightarrow \infty} \sum \limits_{l\in I}  \left| \left\langle U\psi_{i},\phi_{l}\right\rangle \right|^2=0$. 
\item[(2)]  The operator $U$ is in the Schatten $p$-class, if and only if $\mathbf{G}_{U,\Phi,\Psi}$ is Schatten $p$-class. In this case $\left(\sum \limits_{i\in I} \left| \left< U \psi_{i},\phi_{i} \right> \right|^p \right)^{1/p}<\infty$  and $\left\|\left< U \psi_{i},\phi_{l} \right> \right\|_{p,2} < \infty$.
In particular:
\begin{enumerate}
\item[(2a.)] The operator $U$ is trace-class, if and only if $\mathbf{G}_{U,\Phi,\Psi}$ is trace-class, if and only if $\sum \limits_{i\in I} \left| \left\langle U\psi_{i},\phi_{i}\right\rangle \right| < \infty$.
\item[(2b.)] The operator $U$ is Hilbert-Schmidt, if and only if $\mathbf{G}_{U,\Phi,\Psi}$ is Hilbert-Schmidt, if and only if $\sum \limits_{i,l\in I} \left|\left\langle U\psi_{i},\phi_{l}\right\rangle\right|^{2}<\infty$.
\end{enumerate}
 \end{enumerate}
\end{lem}
\begin{proof} This follows from above, 
and Corollary \ref{sec:schattenmat1}.
\end{proof}
This generalizes result for operators and frames \cite{bikhzh15}.
Note that $U$ is compact respectively Schatten $p$-class if and only if $U^*$ is. So, the role of $U$ and $U^*$ as well as $\Phi$ and $\Psi$ can be completely switched (for frames).

\subsection{$U$-cross Gram matrices and Riesz bases}

It is apparent that $\Phi$ is an orthonormal basis if and only if $\mathbf{G}_{\Phi}=I_{\ell^{2}}$ as this means that $\Phi$ is biorthogonal to itself.
In the sequel, we discuss the invertibility of $\mathbf{G}_{U,\Phi,\Psi}$ when $\Phi$ and $\Psi$ are Riesz bases.
\begin{prop}
Let $U\in \BL{\Hil},$ $\Phi=\{\phi_{i}\}_{i\in I}$ and $\Psi=\{\psi_{i}\}_{i\in I}$ be two frames in $\mathcal{H}$ and $\Phi^d$ be a dual of $\Phi$. Then
\begin{enumerate}
\item[(1)]  $\mathbf{G}_{U,\Phi,\Psi}=I_{\ell^{2}}$ if and only if $\Phi$ and $\Psi$ are
 Riesz bases. Also, $\Phi=S_{\Phi}U\Psi$ and $\Psi=S_{\Psi}U^*\Phi.$
 In this case $U=T_{\widetilde{\Phi}} T_{\widetilde{\Psi}}^*$ is invertible. 
\item[(2)]  If $\mathbf{G}_{U,\Phi,\Phi^d}=I_{\ell^{2}}$, then $U=I_{\mathcal{H}}$ and $\Phi^d=\widetilde{\Phi}$. The converse is true only if $\Phi$ is a Riesz basis.
\end{enumerate}
\end{prop}

\begin{proof}
If $\mathbf{G}_{U,\Phi,\Psi}=I_{\ell^{2}}$, then
\begin{eqnarray*}
\delta_{ij}=\left(\mathbf{G}_{U,\Phi,\Psi}\right)_{i,j}=\left\langle U\psi_{j},\phi_{i}\right\rangle.
\end{eqnarray*}
Hence, $\Phi$ has a biorthogonal sequence, and therefore it is a
Riesz basis. Also, $\Psi$ is a Riesz basis since $U^{*}\Phi$ is
its biorthogonal sequence. In particular, $\widetilde{\Phi}=U\Psi$
by Theorem 5.5.4 of \cite{Chr08}. By \eqref{eq:oprec}, $
U=T_{\widetilde\Phi}T_{\widetilde\Psi}^*.$ This shows (1).

 By (1) $\Phi$ is a Riesz basis, and has only one, the canonical dual.
Now, the invertibility of  $S_{\Phi}$ implies that $U=I_{\mathcal{H}}$. The converse is clear.
\end{proof}

In the next theorem, we study sufficient conditions  for the invertibility of the $U$-cross Gram matrix associated to Riesz sequences.
 \begin{thm} \label{sec:invertRiesz1}
 Let $U\in \BL{ \Hil_1,\Hil_2}$, $\Phi=\{\phi\}_{i\in I}$ and $\Psi=\{\psi_{i}\}_{i\in I}$ be two
 Bessel sequences in $\mathcal{H}_2$ and $\Hil_1$, respectively, such that $\mathbf{G}_{U,\Phi,\Psi}$ is
  invertible. Then $\Phi$ and $\Psi$ are Riesz sequences in $\mathcal{H}_2$ and $\Hil_1$, respectively.
    If $\Phi$ and $\Psi$ are assumed to be upper semi-frames, $\Phi$ and $\Psi$ are Riesz bases and $U$ is invertible.
    In this case, \begin{equation*} \label{eq:forminvers}
 \left(\mathbf{G}_{U,\Phi,\Psi}\right)^{-1}=\mathbf{G}_{U^{-1},\widetilde{\Psi},
\widetilde{\Phi}}.
\end{equation*}
 \end{thm}
 \begin{proof}
It is sufficient to show that $T_{\Phi}$
is bounded below.
To see this
 \begin{eqnarray*}
 \left\|d\right\|^{2}&=&\left|\left\langle \mathbf{G}_{U,\Phi,\Psi}\mathbf{G}_{U,\Phi,\Psi}^{-1}d,d\right\rangle\right|\\
 &=&\left|\left\langle T_{\Phi}^{*}UT_{\Psi}\mathbf{G}_{U,\Phi,\Psi}^{-1}d
 ,d\right\rangle\right|\\
 &=&\left|\left\langle T_{\Psi}\mathbf{G}_{U,\Phi,\Psi}^{-1}d,U^{*}T_{\Phi}d
 \right\rangle\right|\\
 &\leq & \sqrt{B_{\Psi}}\left\|\mathbf{G}_{U,\Phi,\Psi}^{-1}\right\|\left\|d\right\|\left\|U^{*}\right\|
 \left\|T_{\Phi}d\right\|,
 \end{eqnarray*}
 for every $d=\{d_{i}\}_{i\in I}\in \ell^{2}$. This follows that
 \begin{eqnarray*}
 \frac{\left\|d\right\|}{\sqrt{B_{\Psi}}\left\|\mathbf{G}_{U,\Phi,\Psi}^{-1}\right\|\left\|U
 \right\|}\leq \left\|T_{\Phi}d\right\|.
 \end{eqnarray*}
 To obtain a lower bound for $\Psi$, an analogue argument can be used.

As $\mathbf{G}_{U,\Phi,\Psi}=T_{\Phi}^* UT_{\Psi}$ it follows that $U$  is invertible for complete sequences.

 \end{proof}
Note that, the invertibility of $\mathbf{G}_{U,\Phi,\Psi}$ does
not imply that $\Phi$ and $\Psi$ are Riesz bases, in general. This is because $\mathbf{G}_{U,\Phi,\Psi}$ can never imply anything about completeness, as the considered space is irrelevant for $\mathbf{G}_{U,\Phi,\Psi}$. For an example assume that
$\{e_{i}\}_{i=1}^{\infty}$ is an orthonormal basis for a
separable Hilbert space $\mathcal{H}$ and $\Phi=\left\{e_{2},e_{3},e_{4},...\right\}$. $\Phi$ is
 non-complete. Still,
\begin{eqnarray*}
\left(\mathbf{G}_{\Phi}\right)_{i,j}=\langle \phi_{j},\phi_{i}\rangle
=\delta_{i,j}.
\end{eqnarray*}
This is even true if one erases countably many elements, for example only considering $\{e_2, e_4, e_6,\dots\}$.

In Theorem \ref{sec:invertRiesz1}, if $\Phi$ and $\Psi$ are Bessel sequences in   finite dimensional\footnote{For finites frames see \cite{xxlfinfram1,caku13}
} Hilbert spaces,
the invertibility
$\mathbf{G}_{U,\Phi,\Psi}$ implies that $\Phi$ and $\Psi$
are Riesz bases and $U$ is invertible operator.
This is because
the invertibility
\begin{eqnarray*}
\mathbf{G}_{U,\Phi,\Psi}=T_{\Phi}^{*}UT_{\Psi}
\end{eqnarray*}
yields $T_{\Phi}^{*}$ is onto and $T_{\Psi}$ is one to one.
Because $\mathcal{H}$ is finite dimensional,  the operators  $T_{\Phi}^{*}$ and
$T_{\Psi}$ are invertible, in particular, $\Phi$ and $\Psi$ are Riesz basis. As a consequence $U$ is also invertible.

The next proposition solves the question of how the above result can be generalized to the existence of a left or right inverses.

\begin{prop}
Let $U\in \BL{ \Hil_1,\Hil_2}$, $\Phi$ and $\Psi$ be Bessel sequences in $\mathcal{H}_2$ and $\Hil_1$, respectively. Then the following assertions are hold.
\begin{enumerate}
\item[(1)]  If $\mathbf{G}_{U,\Phi,\Psi}$ has a right inverse, then $\Phi$ and $U^*\Phi$ are  Riesz sequences. Moreover, if $\Phi$ is  an upper semi-frame, then $\Phi$ is a Riesz basis and $U\Psi$ is a frame.

\item[(2)]  If
$\mathbf{G}_{U,\Phi,\Psi}$ has a left inverse, then $\Psi$ and
$U\Psi$ are  Riesz sequences. Moreover, if $\Psi$ is an upper semi-frame, then
$\Psi$ is a Riesz basis and $U^*\Phi$ is a frame.
\end{enumerate}
\end{prop}
\begin{proof}
(1) The assumption shows that $T_{\Phi}^*U T_{\Psi}=T_{U^*\Phi}^* T_{\Psi}$ is surjective, and so $T_\Phi^*$  and $T_{\Phi}^*U=T_{U^*\Phi}^*$ are surjective.
Using Proposition \ref{Eq.cond.Riesz} immediately follows that $\Phi$ and $U^*\Phi$ are Riesz sequences.
Moreover, if $\Phi$ is an upper semi-frame, then $T_{\Phi}^*$ is bijective by Proposition \ref{Eq.cond.Riesz}, and hence
\begin{eqnarray*}
T_{U\Psi}=\left(T_{\Phi}^*\right)^{-1}T_{\Phi}^*T_{U\Psi}=\left(T_{\Phi}^*\right)^{-1}\mathbf{G}_{U,\Phi,\Psi} 
\end{eqnarray*}
has a  bounded right inverse, or equivalently $U\Psi$ is a frame. The proof of the second part is similar.
\end{proof}

\section{The pseudo-inverse of $U$-cross Gram matrices} \label{sec:pivv}

Similar to the case for multipliers \cite{balsto09new} we can show
that there exist duals that allow the representation of the pseudo-inverse as a matrix of the same class.
Note that, from now, we put as an assumption that the $U$-cross Gram matrix has closed range. In Section \ref{sec:closedrange} we put some statements about when this occurs.

\newcommand{\sd}[2]{#1^{\left(U,#2\right)}}
\newcommand{\ds}[2]{#1^{\overline{\left( U,#2\right)}}}
\newcommand{\sda}[2]{#1^{\left(U^*,#2\right)}}
\newcommand{\dsa}[2]{#1^{\overline{\left( U^*,#2\right)}}}

\def\Phisd{\sd{\Phi}{\Psi}}
\def\phisd{\sd{\phi}{\psi}}

\def\Psipd{\ds{\Psi}{\Phi}}
\def\psipd{\ds{\psi}{\phi}}

\begin{thm} \label{inverse Gram}Let $\Psi$ and $\Phi$ be frames in Hilbert
space  $\mathcal{H}$,
  $U\in \BL{\mathcal{H}}$ be an invertible
operator and $\mathbf{G}_{U,\Phi,\Psi}$ have
 closed range. Then the following assertions hold:
\begin{enumerate}
\item[(1)]  There exists a unique dual $\Phisd$ of $\Phi$ such that
\begin{eqnarray*}
\left(\mathbf{G}_{U,\Phi,\Psi}\right)^{\dagger}=\mathbf{G}_{U^{-1},\widetilde{\Psi},\Phisd}.
\end{eqnarray*}
\item[(2)]There exists a unique dual $\Psipd$ 
 of $\Psi$ such that
\begin{eqnarray*}
\left(\mathbf{G}_{U,\Phi,\Psi}\right)^{\dagger}=\mathbf{G}_{U^{-1},\Psipd,\widetilde{\Phi}}.
\end{eqnarray*}
\end{enumerate}
\end{thm}
\begin{proof}
\begin{enumerate}
\item[(1)]

Note that $\mathbf{G}^{\dag}:=\mathbf{G}_{U,\Phi,\Psi}^{\dag}$
exists and
\begin{eqnarray}\label{salib1}
\kernel{\mathbf{G}^{\dag}}=\left(\range{\mathbf{G}_{U,\Phi,\Psi}}\right)^{\perp}=\left(\range{T_{\Phi}^{*}UT_{\Psi}}\right)^{\perp}=\range{T_{\Phi}^{*}}^{\perp}=
\kernel{T_{\Phi}},
\end{eqnarray}
\begin{eqnarray}\label{salib2}
\range{\mathbf{G}^{\dag}}=\left(\kernel{\mathbf{G}_{U,\Phi, \Psi}}\right)^{\perp}=\left(\kernel{T_{\Phi}^{*}UT_{\Psi}}
\right)^{\perp}=\left(\kernel{T_{\Psi}}
\right)^{\perp}=\range{T_{\Psi}^{*}}.
\end{eqnarray}
\def\phisd{\phi^1{o}}
\def\psisd{\psi^1{o}}
Putting,
\begin{eqnarray}\label{fi}
\Phisd=\{\phi_i^{(U,\Psi)}\}_{i\in
I}:=\{UT_{\Psi}\mathbf{G}^{\dag}\delta_{i}\}_{i\in I},
\end{eqnarray}
where $\{\delta_i\}_{i\in I}$ is the canonical orthonormal basis
of $\ell^2$. Then
\begin{eqnarray*}
T_{\Phisd}T_{\Phi}^*&=&UT_{\Psi}\mathbf{G}^{\dag}T_{\Phi}^{*}\\
 &=&T_{\Phi^d}T_{\Phi}^*UT_{\Psi}\mathbf{G}^{\dag}T_{\Phi}^{*}UT_{\Psi}T_{\Psi^d}^*U^{-1}\\
 &=&T_{\Phi^d}\mathbf{G}_{U,\Phi,\Psi}\mathbf{G}^{\dag}\mathbf{G}_{U,\Phi,\Psi}T_{\Psi^d}^*U^{-1}\\
 &=&T_{\Phi^d}\mathbf{G}_{U,\Phi,\Psi}T_{\Psi^d}^*U^{-1}\\
 &=&T_{\Phi^d}T_{\Phi}^*UT_{\Psi}T_{\Psi^d}^*U^{-1}=I_{\mathcal{H}}.
\end{eqnarray*}

So, $\Phisd$ is a dual of $\Phi$.
Note that for all duals $\Phi^{d}$ and $\Psi^{d}$ of $\Phi$ and $\Psi$, respectively, we have
\begin{eqnarray*}
\mathbf{G}_{U,\Phi,\Psi}\mathbf{G}_{U^{-1},\Psi^{d},\Phi^{d}}\mathbf{G}_{U,\Phi,\Psi}& =&T_{\Phi}^{*}UT_{\Psi}T_{\Psi^d}^{*}U^{-1}T_{\Phi^d}T_{\Phi}^{*}UT_{\Psi}\\
&=&T_{\Phi}^{*}UT_{\Psi}=\mathbf{G}_{U,\Phi,\Psi}.
\end{eqnarray*}

Moreover,
$\kernel{T_{\Phisd}}=\kernel{T_{\Phi}}$. Indeed, by
(\ref{salib1}) and (\ref{fi}) we obtain
 $\kernel{T_{\Phi}}=\kernel{\mathbf{G}^{\dag}}\subseteq
\kernel{T_{\Phisd}}$. For the reverse inclusion, suppose that
$c=\{c_i\}_{i\in I}\in \kernel{T_{\Phisd}}$ and so,
$T_{\Psi}\mathbf{G}^{\dag}c=0.$ On the other hand, by (\ref{salib2})
it follows that
\begin{eqnarray}\label{gsalib}
\mathbf{G^{\dag}}c=T_{\Psi}^*f,
\end{eqnarray}
 for some $f\in \mathcal{H}$. Then
\begin{eqnarray*}
f&=&S_{\Psi}^{-1}T_{\Psi}T_{\Psi}^*f\\
&=&S_{\Psi}^{-1}T_{\Psi}\mathbf{G}^{\dag}c=0.
\end{eqnarray*}
Applying (\ref{gsalib}) and (\ref{salib1}) we have $c\in
\kernel{\mathbf{G}^{\dag}}=\kernel{T_{\Phi}}$.
Furthermore,
\begin{eqnarray*}
\kernel{\mathbf{G}_{U^{-1},\widetilde{\Psi},\Phisd}}&=&
\kernel{T_{\widetilde{\Psi}}^*U^{-1}T_{\Phisd}}\\
&=&\kernel{T_{\Phisd}}\\
&=&\kernel{T_{\Phi}}\\
&=&\kernel{\mathbf{G}^{\dag}}. 
\end{eqnarray*}
Moreover, it follows from (\ref{salib2}) that
\begin{eqnarray*}
\range{\mathbf{G}_{U^{-1}, \widetilde{\Psi},\Phisd}}&=&
\range{T_{ \widetilde{\Psi}}^*U^{-1}T_{\Phisd}} \\
&=&\range{T_{ \widetilde{\Psi}}^*}\\
&=&\range{T_{\Psi}^*}\\
&=&\range{\mathbf{G}^{\dag}}. 
\end{eqnarray*}

Hence, $\mathbf{G}^{\dag}=\mathbf{G}_{U^{-1},\widetilde{\Psi},\Phisd}.$ To
show the uniqueness, assume that $\Phi^{\ddagger}$ is also a dual
of $\Phi$ such that
\begin{eqnarray*}
\mathbf{G}_{U^{-1},\widetilde{\Psi},\Phisd}=\mathbf{G}_{U^{-1},\widetilde{\Psi},\Phi^{\ddagger}}.
\end{eqnarray*}
It follows that  $U^{-1}T_{\Phisd}=U^{-1}T_{\Phi^{\ddagger}}$
and hence, $\Phisd=\Phi^{\ddagger}$. 

The proof of (2) is
similar, using $\ds{\Psi}{\Phi} = \{ \ds{\psi}{\Phi}_i \}_{i\in I}   = \{ U^*T_{\Phi}\left( \mathbf{G}^{\dag} \right)^* \delta_{i} \}_{i\in I}$.
\end{enumerate}
\end{proof}
We have that
$$\ds{\Phi}{\Psi}=\{U^*T_{\Psi}(\mathbf{G}_{U,\Psi,\Phi}^{\dagger})^{*}\delta_i\}_{i\in I}$$
and $(\mathbf{G}_{U,\Psi,\Phi}^{\dagger})^{*}=\mathbf{G}_{U^*,\Phi,\Psi}^{\dagger}$. By comparing $\ds{\Phi}{\Psi}$ and
$$\sd{\Phi}{\Psi}=\{UT_{\Psi}\mathbf{G}_{U,\Phi,\Psi}^{\dagger}\delta_i\}_{i\in I}$$
we obtain that $\sd{\Phi}{\Psi}=\dsa{\Phi}{\Psi}$.

Using the same arguments we can show
\begin{cor} \label{inverseGram2}Let $\Psi$ and $\Phi$ be frames in the Hilbert
spaces  $\mathcal{H}_1$ respectively $\Hil_2$,
  $U\in \BL{\Hil_1,\Hil_2}$  an invertible
operator and $\mathbf{G}_{U,\Phi,\Psi}$ has
 closed range. Then the following assertions hold:
\begin{enumerate}
\item[(1)]  There exists a unique dual $\Phisd$ of $\Phi$ such that
\begin{eqnarray*}
\left(\mathbf{G}_{U,\Phi,\Psi}\right)^{\dagger}=\mathbf{G}_{U^{-1},\widetilde{\Psi},\Phisd}.
\end{eqnarray*}
\item[(2)]There exists a unique dual $\Psipd$ of $\Psi$ such that
\begin{eqnarray*}
\left(\mathbf{G}_{U,\Phi,\Psi}\right)^{\dagger}=\mathbf{G}_{U^{-1},\Psipd,\widetilde{\Phi}}.
\end{eqnarray*}
\end{enumerate}
\end{cor}

Our next goal is to determine $\mathbf{G}_{U,\Phi,\Psi}^{\dag}$ when the invertibility assumption on $U$ is dropped and it is only assumed to be closed range. In fact, we prove that all
results of the above theorem except the uniqueness are true, assuming additionally that $\range{U^*}=S_{\Psi}\range{U^*}$ or $\range{U}=S_{\Phi}\range{U}$.

For that we first look at frames for the range of an operator. Naturally if $\Phi$ is a frame, $\pi_{\range{U}}\Phi = UU^{\dag}\Phi$ is a frame for $\range{U}$.
  Also $U\Phi$ has the same property:
	\begin{cor} \label{cor} Let $U$ have closed range, and $\Psi$ be a frame with bounds $A_\Psi, B_\Psi$.
Then  $U\Psi = \left( U \psi_k \right)_k$ is a frame for $\range{U}$ with frame bounds $m \cdot A_\Psi$, $M \cdot B_\Psi$.  Here, $m$ is the lower bound of $U$, i.e.  $m \norm{}{f}^2 \le \norm{}{U^* f}^2$ for $f \in\kernel{U^*}^{\perp}$  and $M = \norm{}{U^*}^2$.

We have that $S_{U \Psi}^{-1}={\ {U^*}^\dagger}  S_{\Psi}^{-1} U^\dagger$.
\end{cor}
\begin{proof} 
The first part is \cite[Proposition 5.3.1]{Chr08}.

We have $S_{U\Psi} = U\sqrt{S_\Psi} \sqrt{S_\Psi}  U^*$, therefore the
pseudo-inverse is given by 
\begin{eqnarray*}S_{U \Psi}^{-1}&=&\left(S_{U \Psi}\right)^{\dagger}\\
&=&\left( \sqrt{S_\Psi}U^*\right)^{\dagger} \left(  U\sqrt{S_\Psi} \right)^\dagger\\
&=& {U^*}^\dagger \left(\sqrt{S_\Psi}\right)^{-\frac{1}{2}} \left(\sqrt{S_\Psi}\right)^{-\frac{1}{2}} U^\dagger
= {U^*}^\dagger {S_\Psi}^{-1} U^\dagger .
\end{eqnarray*}
\end{proof}
  
\begin{cor}
Let $U\in \BL{\Hil}$ have closed range, $\Phi$ and $\Psi$ be
frames for $\range U$ and $\range {U^*}$, respectively. Then
$\mathbf{G}_{U,\Phi,\Psi}$ has closed range and
$$\left(\mathbf{G}_{U,\Phi,\Psi}\right)^{\dagger}=\mathbf{G}_{U^{\dagger}_{|_{\range {U}}},\Psipd,\widetilde{\Phi}}=\mathbf{G}_{(U_{|_{\range{U^*}}})^{-1},\Psipd,\widetilde{\Phi}}.$$
\end{cor}
\begin{proof}
It follows immediately by using Corollary \ref{inverseGram2} for
the invertible operator  $U:\range {U^*}\rightarrow \range U$
and the fact that  $U^{\dagger}_{|_{\range
{U}}}=(U_{|_{\range{U^*}}})^{-1}$.
\end{proof}

\begin{thm}\label{pseudo Gram}
Let $\Psi$ and $\Phi$ be frames in Hilbert space $\mathcal{H}$,
$U\in \BL{\mathcal{H}}$ a closed range operator and
$\mathbf{G}_{U,\Phi,\Psi}$ have
 closed range. 
\begin{enumerate}
\item[(1)] The following assertions are equivalent:
\item There exists a  dual $\Phisd$ of $\Phi$  on $\range{U}$ such that
\begin{eqnarray*}
\left(\mathbf{G}_{U,\Phi,\Psi}\right)^{\dagger}=\mathbf{G}_{U^{\dag},\widetilde{\Psi},\Phisd}.
\end{eqnarray*}
\item\label{iii}$
 \left(\mathbf{G}_{U,\Phi,\Psi}\right)^{\dagger}=
 \mathbf{G}_{U^{\dagger},\widetilde{\Psi},\widetilde{UU^{\dagger}\Phi}}
$.  \item$\range{U^*}=S_{\Psi}\range{U^*}$.
\end{enumerate}
\begin{enumerate}
\item[(2)] The following assertions are equivalent:
 \item\label{ii} There exists a  dual $\Psipd$ of $\Psi$ on $\range{U^*}$  such that
\begin{eqnarray*}
\left(\mathbf{G}_{U,\Phi,\Psi}\right)^{\dagger}=\mathbf{G}_{U^{\dagger},{\Psipd},\widetilde{\Phi}}.
\end{eqnarray*}
\item $ \left(\mathbf{G}_{U,\Phi,\Psi}\right)^{\dagger}=
\mathbf{G}_{U^{\dagger},\widetilde{U^{\dagger}U\Psi},\widetilde{\Phi}}.$
 \item \label{i}  $\range{U}=S_{\Phi}\range{U}$.
\end{enumerate}
\end{thm}

\begin{proof} 
\begin{enumerate}
For the first part we have

$(1\Leftrightarrow 3)$
Putting $\mathbf{G}^{\dag}:=\left(\mathbf{G}_{U,\Phi,\Psi}\right)^{\dag}$. Then
\begin{eqnarray}\label{salib3}
\kernel{\mathbf{G}^{\dag}}=\left(\range{\mathbf{G}_{U,\Phi,\Psi}}\right)^{\perp}=\left(\range{T_{\Phi}^{*}UT_{\Psi}}\right)^{\perp}=
\range{T_{\Phi}^{*}U}^{\perp}=
\kernel{U^*T_{\Phi}},
\end{eqnarray}
\begin{eqnarray}\label{salib4}
\range{\mathbf{G}^{\dag}}=\left(\kernel{\mathbf{G}_{U,\Phi,\Psi}}\right)^{\perp}=\left(\kernel{T_{\Phi}^{*}UT_{\Psi}}\right)^{\perp}=
 \kernel{UT_{\Psi}}^{\perp}=
\range{T_{\Psi}^*U^*}.
\end{eqnarray}
Take,
\begin{eqnarray}\label{fisalib}
\Phisd=\{{\phi_i}^{(U,\Psi)}\}_{i\in
I}:=\{UT_{\Psi}\mathbf{G}^{\dag}\delta_{i}\}_{i\in I},
\end{eqnarray}
where $\{\delta_i\}_{i\in I}$ is the canonical orthonormal basis
of $\ell^2$. Then $\Phisd$ is a Bessel sequence and on $\range{U}$
we obtain
\begin{eqnarray*}
T_{\Phisd}T_{\Phi}^*&=&UT_{\Psi}\mathbf{G}^{\dag}T_{\Phi}^{*}\\
 &=&T_{\Phi^d}T_{\Phi}^*UT_{\Psi}\mathbf{G}^{\dag}T_{\Phi}^{*}U
 T_{\Psi}T_{\Psi^d}^*U^{\dag}\\
 &=&T_{\Phi^d}\mathbf{G}_{U,\Phi,\Psi}\mathbf{G}^{\dag}\mathbf{G}_{U,\Phi,\Psi}T_{\Psi^d}^*U^{\dag}\\
 &=&T_{\Phi^d}\mathbf{G}_{U,\Phi,\Psi}T_{\Psi^d}^*U^{\dag}\\
 &=& UU^{\dag}=I_{\range{U}},
\end{eqnarray*}
where  $\Phi^{d}$ and $\Psi^{d}$   are duals of $\Phi$ and $\Psi$, respectively.
So, $\Phisd$ is a dual of $\Phi$ on $\range{U}$, in particular a frame on $\range{U}$.
Also,
\begin{eqnarray*}
\mathbf{G}_{U,\Phi,\Psi}\mathbf{G}_{U^{\dag},\Psi^{d},\Phi^{d}}\mathbf{G}_{U,\Phi,\Psi}& =&
T_{\Phi}^{*}UT_{\Psi}T_{\Psi^d}^{*}U^{\dag}T_{\Phi^d}T_{\Phi}^{*}UT_{\Psi}\\
&=&T_{\Phi}^{*}UU^{\dag}UT_{\Psi}\\
&=&T_{\Phi}^{*}UT_{\Psi}=\mathbf{G}_{U,\Phi,\Psi}.
\end{eqnarray*}
Moreover,
$\kernel{U^{\dag}T_{\Phisd}}=\kernel{U^*T_{\Phi}}$. Indeed, the equations
(\ref{salib3}) and (\ref{fisalib}) yield
\begin{eqnarray*}
\kernel{U^*T_{\Phi}}=\kernel{\mathbf{G}^{\dag}}\subseteq
\kernel{T_{\Phisd}}\subseteq \kernel{U^\dag T_{\Phisd}}.
\end{eqnarray*}
  For the reverse inclusion, suppose that
$c=\{c_i\}_{i\in I}\in \kernel{U^{\dag}T_{\Phisd}}$ and so,
$U^{\dag}T_{\Phisd}c=0.$  The injectivity $U^{\dag}$  on $\range{U}$  and  $\range{T_{\Phisd}}\subseteq \range{U}$  imply that  $T_{\Phisd}c=0$. On the other hand, by the fact that  $
\mathbf{G^{\dag}}\mathbf{G}_{U,\Phi,\Psi}\mathbf{G^{\dag}}=\mathbf{G^{\dag}}$
 we have
 \begin{eqnarray*}
 \mathbf{G^{\dag}}c&=&\mathbf{G^{\dag}}\mathbf{G}_{U,\Phi,\Psi}\mathbf{G^{\dag}}c\\
 &=& \mathbf{G^{\dag}}T_{\Phi}^*UT_{\Psi} \mathbf{G^{\dag}}c\\
 &=& \mathbf{G^{\dag}}T_{\Phi}^*T_{\Phisd}c=0.
 \end{eqnarray*}
 Hence, $c\in \kernel{\mathbf{G^{\dag}}}=\kernel{U^*T_{\Phi}}$. Therefore,
\begin{eqnarray*}
\kernel{\mathbf{G}_{U^{\dag},\widetilde{\Psi},\Phisd}}&=&
\kernel{T_{\widetilde{\Psi}}^*U^{\dag}T_{\Phisd}}\\
&=&\kernel{U^{\dag}T_{\Phisd}}\\
&=&   \kernel{U^*T_{\Phi}}=\kernel{\mathbf{G}^{\dag}}.\\
\end{eqnarray*}
Combining (\ref{salib4}) and the assumptions we obtain
\begin{eqnarray*}
\range{\mathbf{G}_{U^{\dag},\widetilde{\Psi},\Phisd}}&=&\range{T_{\widetilde{\Psi}}^*U^{\dag}T_{\Phisd}}\\
&=&\range{T_{\widetilde{\Psi}}^*U^{\dag}}\\
&=&\range{T_{\widetilde{\Psi}}^*U^{*}}\\
  &=&\range{T_{\Psi}^*S_{\Psi}^{-1}U^*}\\
&=&\range{T_{\Psi}^*U^{*}}=\range{\mathbf{G}^{\dag}}.
\end{eqnarray*}
So,
$\mathbf{G}^{\dag}=\mathbf{G}_{U^{\dag},\widetilde{\Psi},\Phisd}.$
 Conversely, suppose there is a  dual of $\Phi$ as $\Phisd$
such that
$\mathbf{G}^{\dag}=\mathbf{G}_{U^{\dag},\widetilde{\Psi},\Phisd}.$
Then
\begin{eqnarray*}
\range{T_{\Psi}^*S_{\Psi}^{-1}U^*}&=&\range{T_{\Psi}^*S_{\Psi}^{-1}U^\dagger}\\
&=&\range{\mathbf{G}_{U^{\dag},\widetilde{\Psi},\Phisd}}\\
&=&\range{\mathbf{G}^{\dag}}\\
&=&\range{\mathbf{G}_{U,\Phi,\Psi}^*}\\
&=& \range{T_{\Psi}^*U^*T_{\Phi}}=\range{T_{\Psi}^*U^*}.
\end{eqnarray*}
This follows that $\range{S_{\Psi}^{-1}U^*}=\range{U^*}.$

($2\Leftrightarrow 3$) It is easy to see that $\widetilde{UU^{\dagger}\Phi}$ is a dual of $\Phi$ on $\range{U}$ and
\begin{eqnarray*} \mathbf{G}_{U,\Phi,\Psi}\mathbf{G}_{U^{\dagger},\widetilde{\Psi},\widetilde{UU^{\dagger}\Phi}}\mathbf{G}_{U,\Phi,\Psi}=\mathbf{G}_{U,\Phi,\Psi}.\end{eqnarray*}
Using this fact $\widetilde{UU^{\dagger}\Phi}$ is a frame on $\range{U}$ and $UU^{\dagger}=\pi_{\range{U}}$ (see Section \ref{sec:closedrange}) we obtain
\begin{eqnarray*}
\kernel{\mathbf{G}_{U^{\dagger},\widetilde{\Psi},\widetilde{UU^{\dagger}\Phi}}}&=&\kernel{T_{\widetilde{\Psi}}^*
U^{\dagger}T_{\widetilde{UU^{\dagger}\Phi}}}\\
&=&\kernel{
U^{\dagger}T_{\widetilde{UU^{\dagger}\Phi}}}\\
&=&\kernel{
T_{\widetilde{UU^{\dagger}\Phi}}}\\
&=&\range{T_{\widetilde{UU^{\dagger}\Phi}}^*}^{\perp}\\
&=&\range{T_{UU^{\dagger}\Phi}^*}^{\perp}\\
&=&\range{T_{\Phi}^*UU^{\dagger}}^{\perp}\\
&=&\range{T_{\Phi}^*U}^{\perp}=\kernel{\mathbf{G}^{\dagger}}.
\end{eqnarray*}
We can see that $\range{U^*}=S_{\Psi}\range{U^*}$ if and only if
\begin{eqnarray*}
\range{\mathbf{G}_{U^{\dagger},\widetilde{\Psi},\widetilde{UU^{\dagger}\Phi}}}&=&
\range{T_{\widetilde{\Psi}}^*
U^{\dagger}T_{\widetilde{UU^{\dagger}\Phi}}}\\
&=&\range{T_{\widetilde{\Psi}}^*
U^{\dagger}}\\
&=&T_{\widetilde{\Psi}}^*
U^{\dagger}(\Hil)\\
&=&T_{\widetilde{\Psi}}^*
U^{*}(\Hil)\\
&=&\range{T_{\Psi}^*
U^{*}}
=\range{\mathbf{G}^{\dagger}}.
\end{eqnarray*}
Hence, $(1)$ is proved.

For the second part note that

 ($1\Leftrightarrow 3$) is similar to the first part.

($2\Leftrightarrow 3$) One can see that $\widetilde{U^{\dagger}U\Psi}$ is a dual of $\Psi$ on $\range{U^{*}}$ and
 $$\mathbf{G}_{U,\Phi,\Psi}\mathbf{G}_{U^{\dagger},\widetilde{U^{\dagger}U\Psi},\widetilde{\Phi}}\mathbf{G}_{U,\Phi,\Psi}=\mathbf{G}_{U,\Phi,\Psi}.$$
Using this fact $\range{(U^{\dagger})^*}=\range{U}$, then $S_{\Phi}\range{U}=\range{U}$ if and only if
 \begin{eqnarray*}
 \kernel{\mathbf{G}_{U^{\dagger},\widetilde{U^{\dagger}U\Psi},\widetilde{\Phi}}}&=&
 \kernel{T_{\widetilde{U^{\dagger}U\Psi}}^*
U^{\dagger}T_{\widetilde{\Phi}}}\\
&=&\kernel{
U^{\dagger}T_{\widetilde{\Phi}}}\\
&=&\kernel{
U^{\dagger}S_{\Phi}^{-1}T_{\Phi}}\\
&=&\range{T_{\Phi}^*S_{\Phi}^{-1}(U^{\dagger})^*}^{\perp}\\
&=&\range{T_{\Phi}^*S_{\Phi}^{-1}U}\\
&=&\range{T_{\Phi}^*U}^{\perp}\\
&=&\kernel{U^*T_{\Phi}}=\kernel{\mathbf{G}^{\dagger}}.
 \end{eqnarray*}
Applying this fact $\widetilde{U^{\dagger}U\Psi}$ is a frame for   $\range{U^{\dagger}}$ we have
 \begin{eqnarray*}
 \range{\mathbf{G}_{U^{\dagger},\widetilde{U^{\dagger}U\Psi},\widetilde{\Phi}}}&=&
 \range{T_{\widetilde{U^{\dagger}U\Psi}}^*
U^{\dagger}T_{\widetilde{\Phi}}}\\
&=&
 \range{T_{\widetilde{U^{\dagger}U\Psi}}^*
U^{\dagger}}\\
&=&
 \range{T_{U^{\dagger}U\Psi}^*
U^{\dagger}}\\
&=&\range{T_{\Psi}^*U^{\dagger}UU^{\dagger}}\\
&=&\range{T_{\Psi}^*U^{\dagger}}\\
&=&\range{T_{\Psi}^*U^{*}}=\range{\mathbf{G}^{\dagger}}.
 \end{eqnarray*}
\end{enumerate}
\end{proof}

The assumptions $\range{U^*}=S_{\Psi}\range{U^*} = \range{S_{\Psi} U^*} $ naturally leads to the question for which operators $U$ this is fulfilled, leading to questions about invariant subspaces, see e.g. \cite{enflom01}, beyond the scope of this paper.

As we have mention before, for a closed range operator  $U$
the uniqueness property of Theorem \ref{inverse Gram} does not hold in general
as the next example indicates.

\begin{ex} Let $\mathcal{H}$ be a Hilbert space with an orthonormal basis $\{e_i\}_{i\in I}$. Let
 $\Psi=\{e_1,e_1,e_2,e_3,e_4,...\}$ and $\Phi=\{e_1,e_2,e_2,e_3,e_4,...\}$. It is clear to see that $\Phi^{a}=\{e_1,0,e_2,e_3,...\}$ and $\Phi^{b}=\{e_1,\frac{e_2}{2},\frac{e_2}{2},e_3,...\}$ are respective duals.
Define  $U\in \BL{\mathcal{H}}$ by
\begin{eqnarray*}
Ue_{i}=e_i, \quad(i\neq 2),\quad Ue_2=e_1.
\end{eqnarray*}
Obviously, $\range{U}=\{e_2\}^{\perp}$ and $\kernel{U}=\finspan{e_1-e_2}$. Hence $U$ has closed
range, so the operator $\mathbf{G}_{U,\Phi,\Psi}$ has also closed
range. Moreover, $U:\kernel{U}^\perp\rightarrow \range{U}$ is invertible, hence $U^\dag$ is given as
\begin{eqnarray*}
U^{\dag}e_{1}=\frac{e_1+e_2}{2},\quad U^{\dag}e_{2}=0,\quad \textrm{and}\quad U^{\dag}e_{i}=e_i, \quad(i\geq 3).
\end{eqnarray*}
In fact,  it is the unique right
inverse of $U$ on $\range{U}$ such that $\range{U^{\dag}}= \range{U^*}$ where $U^*$ is determined by 
\begin{eqnarray*}
U^*e_{1}=e_1+e_2,\quad U^*e_{2}=0,\quad \textrm{and}\quad U^*e_{i}=e_i, \quad(i\geq 3).
\end{eqnarray*}
Moreover,
\begin{eqnarray*}
T_{\Phi^{a}}\{c_k\}=c_1e_1+c_3e_2+c_4e_3+...,
\end{eqnarray*}
\begin{eqnarray*}
T_{\Phi^{b}}\{c_k\}=c_1e_1+c_2e_2/2+c_3e_2/2+c_4e_3+....
\end{eqnarray*}
Hence, $U^{\dag}T_{\Phi^{a}}=U^{\dag}T_{\Phi^{b}}$. So
\begin{eqnarray*}
T_{\Psi^{d}}^*U^{\dag}T_{\Phi^{a}}&=&T_{\Psi^{d}}^*U^{\dag}T_{\Phi^{b}},
\end{eqnarray*}
for every dual $\Psi^{d}$ of $\Psi$. \end{ex}

\begin{cor} Let $U\in \BL{\mathcal{H}}$ be an operator with closed range, and $\Psi$ and $\Phi$  frames in Hilbert space
$\mathcal{H}$.Then 
\begin{eqnarray*}
\left(\mathbf{G}_{U,UU^{\dagger}\Phi,\Psi}\right)^{\dagger}=\left(\mathbf{G}_{U,\Phi,U^{\dagger}U\Psi}\right)^{\dagger}=\mathbf{G}_{U^{\dag},\widetilde{U^{\dagger}U\Psi},\widetilde{UU^{\dagger}\Phi}}.
\end{eqnarray*}
\end{cor}
\begin{proof} By Lemma \ref{lem:closedrang1} $\mathbf{G}_{U,\Phi,\Psi}$ has closed range. 
It is easy to see that $UU^{\dagger}\Phi$ and $U^{\dagger}U\Psi$ are frames  for $\range U$  and $\range{U^*}$, respectively. So, $S_{UU^{\dagger}\Phi}\range U=\range U$ and $S_{U^{\dagger}U\Psi}\range{U^*}
=\range{U^*}$. Then the result immediately follows by Theorem \ref{pseudo Gram}.
\end{proof}

\subsection{More on the closed range conditions} \label{sec:closedrange}

In this subsection we present some conditions for a $U$-cross Gram matrix
having closed range. For example, if $U$ is a positive invertible and
$\Phi$ is a frame, then $\sqrt{U}\Phi$ is a frame, and therefore
$T_{\sqrt{U}\Phi}$ has closed range. Using Corollary 2.3 of
\cite{olepinv} it follows that
$$\mathbf{G}_{U,\Phi,\Phi} = T_\Phi^* U T_\Phi = T_\Phi^* \sqrt{U} \sqrt{U} T_\Phi =
\left(T_{\sqrt{U}\Phi}\right)^*\left(T_{\sqrt{U}\Phi}\right)$$
has closed range.

\begin{prop} Let $U\in \BL{\mathcal{H}}$ have closed range, $\Phi$ be a Bessel sequence  and $\Psi$ a frame for a Hilbert space $\mathcal{H}$.
Then the following are equivalent:
\begin{enumerate}
\item[(1)] \label{r1}  $\mathbf{G}_{U,\Phi,\Psi}$ has closed range.
\item[(2)] \label{r2}$U^*{\Phi}$ is a frame sequence.
\end{enumerate}
\end{prop}

\begin{proof} Since $\Psi$ is a frame for $\mathcal{H}$ we obtain
\begin{equation}\label{Gram cr}
 \range{\mathbf{G}_{U,\Phi,\Psi}}= \range{T_{U^*\Phi}^*T_{\Psi}} =
\range{T_{U^*\Phi}^*}. \end{equation}
The synthesis operator of $U^*{\Phi}$ has closed range \cite{xxlstoeant11,ole1n} 
if and only if $U^*{\Phi}$ is a frame sequence.
\end{proof}
\begin{cor}

Let $U$ be a surjective operator in $\BL{\mathcal{H}}$, $\Phi$  a
Bessel sequence and $\Psi$ a frame for a Hilbert space $\mathcal{H}$. Then the
following are equivalent:
\begin{enumerate}
\item[(1)]   $\mathbf{G}_{U,\Phi,\Psi}$ has closed range. \item[(2)]  ${\Phi}$
is a frame sequence.
\end{enumerate}
\end{cor}

\begin{lem}\label{lem:closedrang1} Let $U\in \BL{\Hil}$ have closed range, and $\Phi$ and $\Psi$ be Bessel sequences in a Hilbert space $\Hil$. The following assertions hold.
\begin{enumerate}
\item[(1)]  If $U \Psi$
 and $U U^\dag \Phi$
 are frames for $\range{U}$, then
$$ \range{\mathbf{G}_{U,\Phi,\Psi}}\ = \range{T_{U U^\dag \Phi}^*}.$$
\item[(2)]  If $U^\dag U \Psi$ 
and $U^* \Phi$ 
are frames for $\range{U^*}$, then
$$ \range{\mathbf{G}_{U,\Phi,\Psi}}\ = \range{T_{U^* \Phi}^*}.$$
\end{enumerate}
In particular, in both cases $\mathbf{G}_{U,\Phi,\Psi}$ has
closed range.

\end{lem}
\begin{proof} We have that $U U^\dag$ is the orthogonal projection on $\range{U}$ \cite{olepinv}. Then
$$ \mathbf{G}_{U,\Phi,\Psi} =  T_\Phi^* U T_\Psi = T_{U U^\dag \Phi}^* U T_\Psi = T_{U U^\dag \Phi}^* T_{U \Psi}.$$
By assumption the considered sequences are frames for $\range{U}$,
and so $\range{\mathbf{G}_{U,\Phi,\Psi}} = \range{T_{U U^\dag
\Phi}^*}.$ This proves the first part. In order to obtain
(2) we have
\begin{eqnarray*}
\mathbf{G}_{U,\Phi,\Psi} =  T_\Phi^* U T_\Psi = T_{\Phi}^* U T_{U^\dag U\Psi} = T_{U^* \Phi}^* T_{U^\dag U\Psi}.
\end{eqnarray*}
\end{proof}
The assumption of Lemma 4.8 are fulfilled, if $\Psi$ and $\Phi$ are frames for $\mathcal{H}$.

\begin{thm}
Let $U\in \BL{\Hil}$ have closed range, $\Phi$ and $\Psi$ be
frames for $\Hil$. Then
$$\left(\mathbf{G}_{U,\Phi,\Psi}\right)^{\dagger}=T_{\widetilde{U\Psi}}^*T_{\widetilde{\Phi}}$$
if and only if $\range{T_{\Phi}^*}=\range{T_{\Phi}^*U}$.
\end{thm}
\begin{proof}
One can see that
\begin{eqnarray*}
\mathbf{G}_{U,\Phi,\Psi}T_{\widetilde{U\Psi}}^*T_{\widetilde{\Phi}}
\mathbf{G}_{U,\Phi,\Psi}&=&T_{\Phi}^*UT_{\Psi}T_{\widetilde{U\Psi}}^*T_{\widetilde{\Phi}}
T_{\Phi}^*UT_{\Psi} \\
&=&T_{\Phi}^*T_{U\Psi}T_{\widetilde{U\Psi}}^*T_{\widetilde{\Phi}}
T_{\Phi}^*UT_{\Psi}\\
&=&T_{\Phi}^*UT_{\Psi}=\mathbf{G}_{U,\Phi,\Psi}.
\end{eqnarray*}
Also,
\begin{eqnarray*}
\range{T_{\widetilde{U\Psi}}^*T_{\widetilde{\Phi}}}&=&
\range{T_{\widetilde{U\Psi}}^*} \\
&=&\range{T_{U\Psi}^*}\\
&=&\range{T_{\Psi}^*U^*}\\
&=&\range{\mathbf{G}_{U,\Phi,\Psi}^*}=\range{\left(\mathbf{G}_{U,\Phi,\Psi}\right)^{\dagger}}.
\end{eqnarray*}
Now,  $\range{T_{\Phi}^*}=\range{T_{\Phi}^*U}$ if and only if
\begin{eqnarray*}
\kernel{T_{\widetilde{U\Psi}}^*T_{\widetilde{\Phi}}}&=&\kernel{T_{\widetilde{\Phi}}}\\
&=&\kernel{T_{\Phi}}\\
&=&\range{T_{\Phi}^*}^{\perp}\\
&=&\range{T_{\Phi}^*U}^{\perp}\\
&=&\range{\mathbf{G}_{U,\Phi,\Psi}}^{\perp}=\kernel{\left(\mathbf{G}_{U,\Phi,\Psi}\right)^{\dagger}}.
\end{eqnarray*}
\end{proof}
\begin{cor}{\label{pseudo3}}
Let $U$ have closed range, and $\Phi$ and $\Psi$ be frames for $\range{U}$ and $\Hil$, respectively. Then
$$\left(\mathbf{G}_{U,\Phi,\Psi}\right)^{\dagger}=T_{\widetilde{U\Psi}}^*T_{\widetilde{\Phi}}.$$
\end{cor}

Based on the above results, we have the following theorem: 
\begin{thm}
Let $U\in\BL{\Hil}$ have closed range, and $\Phi$ and $\Psi$ be
frames for $\Hil$. Then $\mathbf{G}_{U,U\Phi,U^*\Psi}$ has closed
range and
\begin{eqnarray*}
\mathbf{G}_{U_1,U_1\Phi,{U_1}^*\Psi}^{\dagger}=\mathbf{G}_{\left({U_1}^*U_1{U_1}^*\right)^{\dagger},\widetilde{\Psi},\widetilde{\Phi}}=\mathbf{G}_{\left({U_1}^*U_1{U_1}^*\right)^{-1},\widetilde{\Psi},\widetilde{\Phi}},
\end{eqnarray*}
where $U_1=U_{|_{\range{U^*}}}.$
\end{thm}
\begin{proof}
Using Corollary \ref{cor} we have
$S_{U_1{U_1}^*\Psi}^{-1}=\left(U_1{U_1}^*\right)^{\dagger}S_{\Psi}^{-1}
\left(U_1{U_1}^*\right)^{\dagger}$. Applying Corollary \ref{pseudo3}  and the fact that $U_1$ is invertible we
obtain
\begin{eqnarray*}
\mathbf{G}_{U_1,U_1\Phi,{U_1}^*\Psi}^{\dagger}&=&
T_{\widetilde{U_1{U_1}^*}\Psi}^*T_{\widetilde{U_1 \Phi}}\\
&=&T_{U_1{U_1}^*\Psi}^*\left(U_1{U_1}^*\right)^{\dagger}S_{\Psi}^{-1}
\left(U_1{U_1}^*\right)^{\dagger}S_{U_1 \Phi}^{-1} U_1T_{\Phi}\\
&=&T_{\Psi}^* U_1{U_1}^*\left( U_1{U_1}^* \right)^{\dagger}S_{\Psi}^{-1}
\left( U_1{U_1}^*\right)^{\dagger}{ U_1^*}^{\dagger}S_{\Phi}^{-1} U_1^{\dagger} U_1 T_{\Phi}\\
&=&T_{\widetilde{\Psi}}^*\left({U_1}^*U_1{U_1}^*\right)^{\dagger}T_{\widetilde{\Phi}}=
\mathbf{G}_{\left({U_1}^*U_1{U_1}^*\right)^{\dagger},\widetilde{\Psi},\widetilde{\Phi}}.
\end{eqnarray*}

\end{proof}

\section{Approximate duals} \label{sec.apprdual}

We can give several conditions for appropriate duality based on the $U$-cross Gram matrix.We start with sufficient conditions. 
\begin{prop} \label{prop:approxdualsuff}
Let $\Phi$ and $\Psi$ be frames in $\mathcal{H}$ with duals
$\Phi^{d}$ and $\Psi^d,$ respectively. The following assertions
are hold.
\begin{enumerate}
\item[(1)]  $\Phi$ and $\Psi$ are approximate dual frames, if
\begin{eqnarray}\label{app3}
\left\|I_{\ell^{2}}-\mathbf{G}_{\Psi,\Phi}\right\|<
\frac{1}{\sqrt{B_{\Phi}B_{\Phi^d}}}.
\end{eqnarray}
\item[(2)]  $\Phi^d$ and $\Psi^d$ are approximate dual frames, if
\begin{eqnarray}\label{app4}
\left\|I_{\ell^{2}}-\mathbf{G}_{\Phi,\Psi}\right\|<
\frac{1}{\sqrt{B_{\Phi^d}B_{\Psi^d}}}.
\end{eqnarray}
\item[(3)]   $\Phi$ and $\Psi$ are approximate dual frames, if
\begin{eqnarray}\label{app5}
\left\|I_{\ell^{2}}-\mathbf{G}_{\Phi,\Psi}\right\|<
\frac{1}{\sqrt{B_{\Phi}B_{\Phi^d}}}.
\end{eqnarray}
\item[(4)]  If $V\in \BL{\mathcal{H}}$ is a right inverse of $U$ such
that
\begin{eqnarray}\label{app6}
\left\|I_{\ell^{2}}-\mathbf{G}_{U,\Psi,\Phi}\mathbf{G}_{V,\Phi^{d},\Phi}\right\|< \frac{1}{\sqrt{B_{\Phi}B_{\Phi^d}}}, 
\end{eqnarray}
then $\Phi$ and $\Psi$ are approximate dual frames. 
\end{enumerate}
\end{prop}
\begin{proof}

\begin{enumerate}
\item[(1)]  According to the dual property and using (\ref{app3}) we have
\begin{eqnarray*}
\left\|I_{\mathcal{H}}-T_{\Phi}T_{\Psi}^*\right\|&=&\left\|T_{\Phi}\left(I_{\ell^2}-T_{\Psi}^*T_{\Phi}\right)T_{\Phi^d}^*\right\|\\
&\leq&\sqrt{
B_{\Phi}B_{\Phi^d}}\left\|I_{\ell^2}-\mathbf{G}_{\Psi,\Phi}\right\|<1.
\end{eqnarray*}
\item[(2)]  One can see that (\ref{app4}) yields
\begin{eqnarray*}
\left\|I_{\mathcal{H}}-T_{\Phi^d}T_{\Psi^d}^*\right\|&=&\left\|T_{\Phi^d}\left(T_{\Phi}^*T_{\Psi}-I_{\ell^2}\right)T_{\Psi^d}^*\right\|\\
&\leq&\sqrt{
B_{\Phi^d}B_{\Psi^d}}\left\|I_{\ell^2}-\mathbf{G}_{\Phi,\Psi}\right\|<1.
\end{eqnarray*}
\item[(3)] Using (\ref{app5}) it is straightforward to see that
\begin{eqnarray*}
\left\|I_{\mathcal{H}}-T_{\Psi}T_{\Phi}^*\right\|&=&\left\|T_{\Phi^d}\left(I_{\ell^2}-T_{\Phi}^*T_{\Psi}\right)T_{\Phi}^*\right\|\\
&\leq&\sqrt{
B_{\Phi^d}B_{\Phi}}\left\|I_{\ell^2}-\mathbf{G}_{\Phi,\Psi}\right\|<1.
\end{eqnarray*}
\item[(4)] Finally, note that
$\mathbf{G}_{U\Psi,\Phi}\mathbf{G}_{V,\Phi^{d},\Phi}=\mathbf{G}_{\Psi,\Phi}$.
Then the result follows immediately from (\ref{app6}) and the first part. 

 \end{enumerate}
\end{proof}
Note that the role in this result of the primal and dual frames, i.e. $\Psi$, $\Phi$ and $\Psi^d$, $\Phi^d$ can be switched.

We can only give one necessary condition, and this holds only in the Riesz basis case:
\begin{lem}  If $\Psi$ is  an  approximate dual for a Riesz basis $\Phi$, then
\begin{eqnarray*}
\left\|I_{\ell^2}-\mathbf{G}_{\Phi, \Psi}\right\|<\sqrt{\frac{B_{\Phi}B_{\Psi}}{A_{\Phi}A_{\Psi}}}.
\end{eqnarray*}
\end{lem}
\begin{proof}
   Let $\Psi$ be a Riesz basis and $\Phi$ an approximate dual $\Psi$. Then $T_{\Psi}$ and $\mathbf{G}_{\Phi,\Psi}$ are invertible. It follows that
\begin{eqnarray*}
\left\|(T_{\Phi}T_{\Psi}^*)^{-1}\right\|^{-1}\left\|I_{\ell^2}-T_{\Phi}^*T_{\Psi}\right\|&\leq& \left\|T_{\Phi}^*T_{\Psi}\left(I_{\ell^2}-T_{\Phi}^*T_{\Psi}\right)\right\|\\
&\leq&\left\|T_{\Phi}^*T_{\Psi}-T_{\Phi}^*T_{\Psi}T_{\Phi}^*T_{\Psi}\right\|\\
&<&\left\|T_{\Phi}^*\left(I_{\mathcal{H}}-T_{\Psi}T_{\Phi}^*\right)T_{\Psi} \right\|\\
&\leq& \sqrt{B_{\Phi}}\sqrt{B_{\Psi}}.
\end{eqnarray*}
So,
\begin{eqnarray*}
\left\|I_{\ell^2}-T_{\Phi}^*T_{\Psi}\right\|\leq \sqrt{B_{\Phi}}\sqrt{B_{\Psi}}\left\|(T_{\Phi}T_{\Psi}^*)^{-1}\right\|\\
\leq\sqrt{B_{\Phi}}\sqrt{B_{\Psi}}\left\|T_{\Phi}^{-1}\right\|\left\|T_{\Psi}^{-1}\right\|\leq \sqrt{\frac{B_{\Phi}B_{\Psi}}{A_{\Phi}A_{\Psi}}}.
\end{eqnarray*}
\end{proof}

 \section{Stability of $U$-cross Gram matrices} \label{sec:stability}
 In this section, we  state some sufficient  conditions for the invertibility of  $U$-cross Gram matrices.

\begin{prop}
Let $\Phi$ be a Bessel sequence in $\mathcal{H}$  with Bessel
bound $B_{\Phi}$. If $U_{1},U_{2}$ and $U_{3}\in \BL{\mathcal{H}}$ such
that $\mathbf{G}_{U_{1},\Phi,\Phi}$ is an invertible operator  and

 \begin{eqnarray}\label{2}
 \left\|U_{2}^{*}U_{1}U_{3}-U_{1}\right\|<\frac{1}{\left\|\mathbf{G}_{U_{1},\Phi,\Phi}^{-1}\right\|B_{\Phi}},
 \end{eqnarray}
 then $\mathbf{G}_{U_{1}, U_{2}\Phi,U_{3}\Phi}$ is also invertible. Moreover, if $\Phi$ is a frame then $\Phi$ is a Riesz basis, $U_{1}$ is invertible and
 \begin{eqnarray*}
  \mathbf{G}_{U_{1},U_2\Phi,U_3\Phi}^{-1}=T_{\widetilde{\Phi}}^{*}\sum_{k=0}^{\infty}
  \left(I_{\mathcal{H}}-U_{1}^{-1}U_{2}^{*}U_{1}U_{3}\right)^{k}U_{1}^{-1}T_{\widetilde{\Phi}},
 \end{eqnarray*}
\end{prop}
\begin{proof}
 Assumption (\ref{2}) yields
\begin{eqnarray*}
 \left\|\mathbf{G}_{U_{1},U_{2}\Phi,U_{3}\Phi}\mathbf{G}_{U_{1},\Phi,\Phi}^{-1}-I_{\ell^{2}}\right\|&=&\left\|\left(\mathbf{G}_{U_{1},U_{2}\Phi,U_{3}\Phi}-\mathbf{G}_{U_{1},\Phi,\Phi}\right)\mathbf{G}_{U_{1},\Phi,\Phi}^{-1}\right\|\\
&\leq&\left\|\mathbf{G}_{U_{2}^{*}U_{1}U_{3},\Phi,\Phi}-\mathbf{G}_{U_{1},\Phi,\Phi}\right\|\left\|\mathbf{G}_{U_{1},\Phi,\Phi}^{-1}\right\|\\
&=&\left\|T_{\Phi}^{*}\left(U_{2}^{*}U_{1}U_{3}-U_{1}\right)T_{\Phi}\right\|\left\|\mathbf{G}_{U_{1},\Phi,\Phi}^{-1}\right\|\\
&\leq&\left\|T_{\Phi}\right\|^{2}\left\|(U_{2}^{*}U_{1}U_{3}-U_{1}\right\|\left\|\mathbf{G}_{U_{1},\Phi,\Phi}^{-1}\right\|< 1.
 \end{eqnarray*}
This shows that $\mathbf{G}_{U_{1},U_{2}\Phi,U_{3}\Phi}\mathbf{G}_{U_{1},\Phi,\Phi}^{-1}$ is invertible and hence,  $\mathbf{G}_{U_{1},U_{2}\Phi,U_{3}\Phi}$ is invertible. Moreover, if $\Phi$ is a
frame, then it is also a Riesz basis, $U_1$ is invertible and $\mathbf{G}_{U_{1},\Phi,\Phi}^{-1}=T_{\Phi}^{-1}U_1^{-1}\left(T_{\Phi}^*\right)^{-1}$. Due to Proposition \ref{invOP} we obtain
 \begin{eqnarray*}
 \mathbf{G}_{U_{1},U_{2}\Phi,U_{3}\Phi}^{-1}&=&
 \sum_{k=0}^{\infty}\left(\mathbf{G}_{U_{1},\Phi,\Phi}^{-1}\left(\mathbf{G}_{U_{1},\Phi,\Phi}- \mathbf{G}_{U_{1},U_{2}\Phi,U_{3}\Phi}\right)\right)^{k}\mathbf{G}_{U_{1},\Phi,\Phi}^{-1}\\
 &=&\sum_{k=0}^{\infty}\left(T_{\Phi}^{-1}U_1^{-1}\left(T_{\Phi}^*\right)^{-1}
 \left(T_{\Phi}^{*}\left(U_1-U_{2}^{*}U_1U_{3}\right)T_{\Phi}\right)\right)^{k}\mathbf{G}_{U_{1},\Phi,\Phi}^{-1}\\
  &=&\sum_{k=0}^{\infty}\left(T_{\Phi}^{-1}U_1^{-1}
 \left(U_1-U_{2}^{*}U_1U_{3}\right)T_{\Phi}\right)^{k}
 T_{\Phi}^{-1}U_1^{-1}\left(T_{\Phi}^*\right)^{-1}\\
&=&\sum_{k=0}^{\infty}T_{\Phi}^{-1}
\left(I_{\mathcal{H}}-U_{1}^{-1}U_{2}^*U_{1}U_3)\right)^{k}T_{\Phi} T_{\Phi}^{-1}U_{1}^{-1}\left(T_{\Phi}^{*}\right)^{-1}\\
 &=&T_{\widetilde{\Phi}}^{*}\sum_{k=0}^{\infty}
\left(I_{\mathcal{H}}-U_{1}^{-1}U_{2}^*U_{1}U_3\right)^{k}U_{1}^{-1}T_{\widetilde{\Phi}}.
 \end{eqnarray*}
\end{proof}
\begin{cor}
Let $\Phi$ be a Bessel sequence in $\mathcal{H}$  with Bessel
bound $B_{\Phi}$. If $U_{1},U_{2}\in \BL{\mathcal{H}}$ such that
$\mathbf{G}_{U_{1},\Phi,\Phi}$ is an invertible operator  and
\begin{eqnarray}\label{1}
\left\|U_{2}-I_{\mathcal{H}}\right\|< \frac{1}{\left\|\mathbf{G}_{U_{1},\Phi,\Phi}^{-1}\right\|B_{\Phi}\|U_{1}\|},
\end{eqnarray}
then   $\mathbf{G}_{U_{1},\Phi,U_{2}\Phi}$  and
$\mathbf{G}_{U_{1},U_{2}\Phi,\Phi}$
 are also invertible.  Moreover, if $\Phi$ is a frame and $U_{1}\in \BL{\mathcal{H}}$ is invertible, then
 \begin{eqnarray*}
 \mathbf{G}_{U_{1},\Phi,U_2\Phi}^{-1}=\mathbf{G}_{U_1^{-1},\widetilde{U_2\Phi},\widetilde{\Phi}}
 \end{eqnarray*} and
 \begin{eqnarray*}
 \mathbf{G}_{U_{1},U_2\Phi,\Phi}^{-1}=\mathbf{G}_{U_1^{-1},\widetilde{\Phi},\widetilde{U_2\Phi}}.
 \end{eqnarray*}
\end{cor}
\begin{proof}
 By using the assumption (\ref{1}) we obtain
\begin{eqnarray*}
\left\|\mathbf{G}_{U_{1},\Phi,U_{2}\Phi}\mathbf{G}_{U_{1},\Phi,\Phi}^{-1}- I_{\ell^{2}}\right\|&=&\left\|\left(\mathbf{G}_{U_{1},\Phi,U_{2}\Phi}-\mathbf{G}_{U_{1},\Phi,\Phi}\right)\mathbf{G}_{U_{1},\Phi,\Phi}^{-1}\right\|\\
&\leq&\left\|\mathbf{G}_{U_{1},\Phi,U_{2}\Phi}-\mathbf{G}_{U_{1},\Phi,\Phi}\right\|\left\|\mathbf{G}_{U_{1},\Phi,\Phi}^{-1}\right\|\\
&=&\left\|T_{\Phi}^{*}U_{1}\left(U_{2}-I_{\mathcal{H}}\right)T_{\Phi}\right\|\left\|\mathbf{G}_{U_{1},\Phi,\Phi}^{-1}\right\|\\
&\leq& \left\|T_{\Phi}\right\|^{2}\|U_{1}\|\left\|U_{2}-I_{\mathcal{H}}\right\|\left\|\mathbf{G}_{U_{1},\Phi,\Phi}^{-1}\right\|\\
&\leq&B_{\Phi}\|U_{1}\|\left\|U_{2}-I_{\mathcal{H}}\right\|\left\|\mathbf{G}_{U_{1},\Phi,\Phi}^{-1}\right\|< 1.
\end{eqnarray*}
Then $\mathbf{G}_{U_{1},\Phi,U_{2}\Phi}\mathbf{G}_{U_{1},\Phi,\Phi}^{-1}$ is invertible and so $\mathbf{G}_{U_{1},\Phi,U_{2},\Phi}$ is invertible.
The proof of invertibility $\mathbf{G}_{U_{1},U_{2}\Phi,\Phi}$ is similar. The rest is immediately follows by Theorem \ref{sec:invertRiesz1}.
\end{proof}
 \begin{thm}
 Suppose that $\Phi$ and $\Psi$ are  Bessel sequences in $\mathcal{H}$  such that $\mathbf{G}_{U,\Phi,\Psi}$ is invertible.
 \begin{enumerate}
 \item[(1)] If $V\in \BL{\mathcal{H}}$ such that
\begin{eqnarray}\label{C1}
 \left\|U-V\right\|<\frac{1}{\left\|\mathbf{G}_
 {U,\Phi,\Psi}^{-1}\right\|\sqrt{B_{\Phi}B_{\Psi}}},
\end{eqnarray}
 then $\mathbf{G}_{V,\Phi,\Psi}$ is also invertible.
\item[(2)] If $\Xi=\{\xi_{i}\}_{i\in I}$ is a Bessel
    sequence in $\mathcal{H}$ such that
\begin{eqnarray}\label{C2}
\left(\sum_{i\in I}\left\|\psi_{i}-\xi_{i}\right\|^{2}\right)^{1/2}< \frac{1}{\left\|\mathbf{G}_
 {U,\Phi,\Psi}^{-1}\right\|\sqrt{B_{\Phi}}\left\|U\right\|},
\end{eqnarray}
 then $\mathbf{G}_{U,\Phi,\Xi}$ is invertible.

\item[(3)] If $\Theta=\{\theta_{i}\}_{i\in I}$ is a
    Bessel sequence in $\mathcal{H}$ such that
    \begin{eqnarray*}\label{C3}
\left(\sum_{i\in I}\left\|\phi_{i}-\theta_{i}\right\|^{2}\right)^{1/2}< \frac{1}{\left\|\mathbf{G}_
 {U,\Phi,\Psi}^{-1}\right\|\sqrt{B_{\Psi}}\left\|U\right\|},
    \end{eqnarray*}
    then $\mathbf{G}_{U,\Theta,\Psi}$ is invertible.
 \end{enumerate}
 \end{thm}
 \begin{proof}  By assumption (\ref{C1}) we have
 \begin{eqnarray*}
  \left\|I_{\ell^{2}}-\mathbf{G}_{U,\Phi,\Psi}^{-1}\mathbf{G}_{V,\Phi,\Psi}\right\|&=&\left\|\mathbf{G}_{U,\Phi,\Psi}^{-1}\left(\mathbf{G}_{U,\Phi,\Psi}-\mathbf{G}_{V,\Phi,\Psi}\right)\right\|\\
&\leq& \left\|\mathbf{G}_{U,\Phi,\Psi}^{-1}\right\| \left\|\mathbf{G}_{U,\Phi,\Psi}-\mathbf{G}_{V,\Phi,\Psi}\right\|\\
&=&\left\|\mathbf{G}_{U,\Phi,\Psi}^{-1}\right\|\left\|T_{\Phi}^{*}
 UT_{\Psi}-T_{\Phi}^{*}VT_{\Psi}\right\|\\
 &=&\left\|\mathbf{G}_{U,\Phi,\Psi}^{-1}\right\|\left\|T_{\Phi}^{*}(U-V)T_{\Psi}\right\|\\
 &\leq& \left\|\mathbf{G}_{U,\Phi,\Psi}^{-1}\right\|\sqrt{B_{\Phi}B_{\Psi}}\|U-V\|< 1,
 \end{eqnarray*}
 and  $\mathbf{G}_{V,\Phi,\Psi}$ is also invertible. This proves (1). To show (2) note that
 \begin{eqnarray}\label{Tsi}
 \left\|T_{\Psi}-T_{\Xi}\right\|\leq \left(\sum_{i\in I}\left\|\psi_{i}-\xi_{i}\right\|^{2}\right)^{1/2}.
 \end{eqnarray}
 Using (\ref{C2}) follows that
  \begin{eqnarray*}
  \left\|I_{\ell^{2}}-\mathbf{G}_{U,\Phi,\Psi}^{-1}\mathbf{G}_{U,\Phi,\Xi}\right\|&=&\left\|\mathbf{G}_{U,\Phi,\Psi}^{-1}\left(\mathbf{G}_{U,\Phi,\Psi}-\mathbf{G}_{U,\Phi,\Xi}\right)\right\|\\
 &\leq& \left\|\mathbf{G}_{U,\Phi,\Psi}^{-1}\right\|\left\|\mathbf{G}_{U,\Phi,\Psi}-\mathbf{G}_{U,\Phi,\Xi}\right\|\\
 &=&\left\|\mathbf{G}_{U,\Phi,\Psi}^{-1}\right\|
 \left\|T_{\Phi}^{*}UT_{\Psi}-T_{\Phi}^{*}UT_{\Xi}\right\|\\
 &=&\left\|\mathbf{G}_{U,\Phi,\Psi}^{-1}\right\|\left\|T_{\Phi}^{*}U(T_{\Psi}-T_{\Xi})\right\|\\
&\leq& \left\|\mathbf{G}_{U,\Phi,\Psi}^{-1}\right\|\sqrt{B_{\Phi}}\|U\|\left(\sum_{i\in I}\left\|\psi_{i}-\xi_{i}\right\|^{2}\right)^{1/2}
< 1.
 \end{eqnarray*}  Hence, $\mathbf{G}_{U,\Phi,\Xi}$ is invertible by the invertibility $\mathbf{G}_{U,\Phi,\Psi}^{-1}\mathbf{G}_{U,\Phi,\Xi}$.
Finally,
 (3) follows similarly.
  \end{proof}
 
Note that the condition $\left(\sum_{i\in I}\left\|\psi_{i}-\xi_{i}\right\|^{2}\right)^{1/2}$ is a typical condition for results dealing with  the perturbation of frames \cite{ole1n} or 'nearness of sequences' \cite{aria07,xxlframoper1}.

 \begin{thm}
 Let $\Psi=\{\psi_{i}\}_{i\in I}$  be a Bessel sequence and $\Phi=\{\phi_{i}\}_{i\in I}$   a Riesz basis  such that
 \begin{eqnarray*}
\sum_{i\in I}\left\|U\psi_{i}-\phi_{i}\right\|^{2}< \frac{A_{{\Phi}}^2}{B_{{\Phi}}},
 \end{eqnarray*}
 where $A_{{\Phi}}$ and $B_{{\Phi}}$ are lower and upper bounds of $\Phi$, respectively.
 Then $\mathbf{G}_{U,\Phi,\Psi}$ is invertible  and
 \begin{eqnarray*}
 \mathbf{G}_{U,\Phi,\Psi}^{-1}=\sum_{k=0}^{\infty}\left(I_{\ell^{2}}-T_{\Phi}^{-1}UT_{\Psi}\right)^{k}\mathbf{G}_{\Phi}^{-1}.
 \end{eqnarray*}
 \end{thm}
 \begin{proof}
Since $\Phi$ is a Riesz basis, we conclude that $\mathbf{G}_{\Phi}$ is invertible and
\begin{eqnarray*}\left\|\mathbf{G}_{\Phi}^{-1}\right\|=\left\|T_{\Phi}^{-1}(T_{\Phi}^{*})^{-1}\right\|\leq A_{{\Phi}}^{-1}.
\end{eqnarray*}
 Therefore,
 \begin{eqnarray*}
 \left\|\mathbf{G}_{U,\Phi,\Psi}-\mathbf{G}_{\Phi}\right\|&=&
 \left\|T_{\Phi}^{*}UT_{\Psi}-T_{\Phi}^{*}T_{\Phi}\right\|\\
 &=&\sqrt{B_{{\Phi}}}\left\|UT_{\Psi}-T_{\Phi}\right\|\\
&\leq&\sqrt{B_{{\Phi}}}\left(\sum_{i\in I}\left\|U\psi_{i}-\phi_{i}\right\|^{2}\right)^{1/2}\\
&\leq&A_{{\Phi}}\ \leq \left\|\mathbf{G}_{\Phi}^{-1}\right\|^{-1}.
 \end{eqnarray*}
Hence, $\mathbf{G}_{U,\Phi,\Psi}$ is invertible by  Proposition \ref{invOP}.  Moreover, by Proposition \ref{invOP} we have
\begin{eqnarray*}
\mathbf{G}_{U,\Phi,\Psi}^{-1}&=&\sum_{k=0}^{\infty}\left(\mathbf{G}_{\Phi}^{-1}
\left(\mathbf{G}_{\Phi}-\mathbf{G}_{U,\Phi,\Psi}\right)\right)^{k}\mathbf{G}_{\Phi}^{-1}\\
&=&\sum_{k=0}^{\infty}\left(\mathbf{G}_{\Phi}^{-1}T_{\Phi}^{*}
\left(T_{\Phi}-UT_{\Psi}\right)\right)^{k}\mathbf{G}_{\Phi}^{-1}\\
&=&\sum_{k=0}^{\infty}
\left(I_{\ell^2}-T_{\Phi}^{-1}UT_{\Psi}\right)^{k}\mathbf{G}_{\Phi}^{-1}.
\end{eqnarray*}
 \end{proof}

Now we are ready to state our main result about the stability of
$U$-cross Gram matrices.
\begin{thm}
Let $U$ and $V\in \BL{\mathcal{H}},$  $\Phi=\{\phi_{i}\}_{i\in I}$ and $\Psi=\{\psi_{i}\}_{i\in I}$
be frames. Let
$\mathbf{G}_{U,\Phi,\Psi}$ be invertible. If
$\Xi=\{\xi_{i}\}_{i\in I}$ and $\Theta=\{\theta_{i}\}_{i\in I}$
are Bessel sequences such that
\begin{eqnarray}\label{shart1}
\left\|\sum_{i\in I}c_{i}\left(\psi_{i}-\theta_{i}\right)\right\|+\left\|\sum_{i\in I}c_{i}\left(\phi_{i}-\xi_{i}\right)\right\|&\leq&\lambda_{1}\left\|\sum_{i\in I}c_{i}\psi_{i}\right\|
+\lambda_{2}\left\|\sum_{i\in I}c_{i}\phi_{i}\right\|\\  
&&+\lambda_{3}\left\|\sum_{i\in I}c_{i}\xi_{i}\right\|+\lambda_{4}\left\|\sum_{i\in I}c_{i}\theta_{i}\right\|, \nonumber 
\end{eqnarray} 
 for all $c=\{c_{i}\}_{i\in I}\in \ell^{2}$, and
\begin{equation}\label{shart2}
\left\|U-V\right\|<\mu,\qquad \mu+2\|U\|\lambda<\frac{\sqrt{A_{\Psi}A_{\Phi}}}{\left\|U^{-1}\right\|B}, \  \text{and}\quad
\lambda\left(1+3\sqrt{\frac{B}{A}}\right)<1,
\end{equation}
where $B=max\{B_{\Phi},B_{\Psi},B_{\Xi},B_{\Theta}\}$,
$\lambda=\lambda_{1}+\lambda_{2}+\lambda_{3}+\lambda_{4}$ and
$A_{\Psi}$, $A_{\Phi}$ are lower bounds of $\Psi$ and $\Phi$,
respectively. Then $\mathbf{G}_{V,\Xi,\Theta}$ is invertible and
$\Xi$ and $\Theta$ are Riesz bases.
\end{thm}
\begin{proof}
First note that from (\ref{shart1}) it easily follows that
\begin{equation}\label{T}
\left\|T_{\Psi}-T_{\Theta}\right\|+\left\|T_{\Phi}-T_{\Xi}\right\|\leq
\lambda_{1}\sqrt{B_{\Psi}}+\lambda_{2}\sqrt{B_{\Theta}}+
\lambda_{3}\sqrt{B_{\Phi}}+\lambda_{4}\sqrt{B_{\Xi}}.
\end{equation}
This immediately implies that $\Xi$ and $\Theta$ are frames
by Theorem 5.6.1 of \cite{Chr08}.
On the other hand, Theorem
\ref{sec:invertRiesz1} implies that $\Phi$ and $\Psi$  are Riesz bases and $U$
is invertible. In particular,
\begin{equation*}
\mathbf{G}_{U,\Phi,\Psi}^{-1}=
\mathbf{G}_{U^{-1},\widetilde{\Psi},\widetilde{\Phi}}.
\end{equation*}
This shows that
\begin{equation}\label{shart3}
\frac{\sqrt{A_{\Psi}A_{\Phi}}}{\left\|U^{-1}\right\|}\leq \frac{1}{\left\|\mathbf{G}_{U,\Phi,\Psi}^{-1}\right\|},
\end{equation}
where $A_{\Psi}$ and $A_{\Phi}$ are the lower frame bounds of
$\Psi$ and $\Phi$, respectively. Combining  (\ref{shart2}),
(\ref{T}) and (\ref{shart3}) we obtain
\begin{eqnarray*}
\left\|\mathbf{G}_{V,\Xi,\Theta}-\mathbf{G}_{U,\Phi,\Psi}\right\|&=&\left\|\mathbf{G}_{V,\Xi,\Theta}-\mathbf{G}_{U,\Xi,\Theta}+\mathbf{G}_{U,\Xi,\Theta}-\mathbf{G}_{U,\Xi,\Psi}+\mathbf{G}_{U,\Xi,\Psi}-\mathbf{G}_{U,\Phi,\Psi}\right\|\\
&\leq&\left\|T_{\Xi}^{*}(V-U)T_{\Theta}\right\|+\left\|T_{\Xi}^{*}U(T_{\Theta}-T_{\Psi})\right\|+\left\|(T_{\Phi}^{*}-T_{\Xi}^{*})UT_{\Psi}\right\|\\
&\leq&\mu\sqrt{B_{\Xi}B_{\Theta}}+\sqrt{B_{\Xi}}\|U\|\left\|T_{\Theta}-T_{\Psi}\right\|+\|U\|\sqrt{B_{\Psi}}\left\|T_{\Phi}-T_{\Xi}\right\|\\
&\leq&\mu\sqrt{B_{\Xi}B_{\Theta}}+(\sqrt{B_{\Xi}}+\sqrt{B_{\Psi}})\|U\|(\lambda_{1}\sqrt{B_{\Psi}}+\lambda_{2}\sqrt{B_{\Theta}}\\
&&+
\lambda_{3}\sqrt{B_{\Phi}}+\lambda_{4}\sqrt{B_{\Xi}})\\
&\le & B\left(\mu+2\|U\|\lambda\right) \\
&\leq& \frac{\sqrt{A_{\Psi}A_{\Phi}}}{\left\|U^{-1}\right\|}\leq \left\|\mathbf{G}_{U,\Phi,\Psi}^{-1}\right\|^{-1}.
\end{eqnarray*}

Hence, $\mathbf{G}_{V,\Xi,\Theta}$ is invertible by  Proposition
\ref{invOP}.  In particular, $\Xi$ and $\Theta$ are Riesz
bases by Theorem \ref{sec:invertRiesz1}.
\end{proof}
\begin{cor}
Let $U\in \BL{\mathcal{H}},$ $\Phi=\{\phi_{i}\}_{i\in I}$ and $\Psi=\{\psi_{i}\}_{i\in I}$
be frames in $\Hil$.
$\Phi^{n}=\{\phi_{i}^{n}\}_{i\in I}\rightarrow
\Phi$ \footnote{$\Phi^{n}=\{\phi_{i}^{n}\}_{i\in I}\rightarrow
\Phi$ if $\forall \epsilon > 0$ $\exists N$ such that $\sum
\limits_{i \in I} \norm{}{\phi_i^n - \phi_i}^{ 2} < \epsilon$
for all $n \ge N$}, $\Psi^{n}=\{\psi_{i}^{n}\}_{i\in I}\rightarrow
\Psi$ in {$\mathcal{H}$} and $U_n\rightarrow U$ in $\BL{
\mathcal{H}}$, then $\mathbf{G}_{U_n,\Phi^n,\Psi^n}\rightarrow
\mathbf{G}_{U,\Phi,\Psi}$ in $\BL {\ell^{2}}$.
\end{cor}
\begin{proof}
Applying (\ref{Tsi}) and assumptions we have
\begin{eqnarray*}
\left\|\mathbf{G}_{U_n,\Phi^n,\Psi^n}-\mathbf{G}_{U,\Phi,\Psi}\right\|&=&\left\|T_{\Phi^{n}}^*U_nT_{\Psi^n}-T_{\Phi}^*UT_{\Psi}\right\|\\
&\leq&\left\|T_{\Phi^{n}}^*U_nT_{\Psi^n}-T_{\Phi^{n}}^*U_nT_{\Psi}\right\|
+\left\|T_{\Phi^{n}}^*U_nT_{\Psi}-
T_{\Phi}^*UT_{\Psi}\right\|\\
&\leq&\left\|T_{\Phi^{n}}^*U_n\right\|\left\|T_{\Psi^n}-T_{\Psi}\right\|+\left\|T_{\Phi^{n}}^*U_n-T_{\Phi}^*U\right\|\left\|T_{\Psi}\right\|\\
&\leq&\left\|T_{\Phi^{n}}^*U_n\right\|\left\|T_{\Psi^n}-T_{\Psi}\right\|\\
&&+\left(\left\|T_{\Phi^{n}}^*U_n-T_{\Phi}^*U_n\right\|+\left\|T_{\Phi}^*U_n-
T_{\Phi}^*U\right\|\right)\left\|T_{\Psi}\right\|\\
&\leq&\left\|T_{\Phi^{n}}^*U_n\right\|\left\|T_{\Psi^n}-T_{\Psi}\right\|\\
&&+\left(\left\|T_{\Phi^{n}}^*-T_{\Phi}^*\right\|\left\|U_n\right\|+\left\|T_{\Phi}^*\right\|\left\|U_n-
U\right\|\right)\left\|T_{\Psi}\right\|
\rightarrow 0.
\end{eqnarray*}
\end{proof}

\section*{Acknowledgement}
The first and fourth authors were partly supported by the START
project FLAME Y551-N13 of the Austrian Science Fund (FWF). P.B.
was also partly supported by the DACH project BIOTOP I-1018-N25 of
Austrian Science Fund (FWF). He wishes to thank NuHAG for the
availability of its webpage. He also thanks Nora Simovich for help
with typing.


\begin{thebibliography}{10}



 \bibitem{jpaxxl09}
J.-P.~ Antoine and P. Balazs.
\newblock  Frames and semi-frames.
\newblock {\em J. Phys. A}, 44:205201, 2004.





\bibitem{antbal12}
J.-P.~ Antoine and P. Balazs.
\newblock  Frames, semi-frames, and {H}ilbert scales.
\newblock {\em Numer. Funct. Anal. Optim.}, 33(7-9):736--769, 2012.


\bibitem{aria07}
M.~Arias, G.~Corach and M.~Pacheco.
\newblock  Characterization of {B}essel sequences.
\newblock {\em Extracta Math.}, 22(1):55--66, 2007.


\bibitem{xxlmult1}
P.~Balazs.
\newblock  Basic definition and properties of {B}essel multipliers.
\newblock {\em J. Math. Anal. Appl.}, 325(1):571--585,  2007.



\bibitem{xxlfinfram1}
P.~Balazs.
\newblock  Frames and finite dimensionality: Frame transformation,
  classification and algorithms.
\newblock {\em  Appl. Math. Sci.}, 2(41-44):2131--2144, 2008.


\bibitem{xxlframoper1}
P.~Balazs.
\newblock  Matrix-representation of operators using frames.
\newblock {\em Sampl. Theory Signal Image Process.}, 
 7(1):39--54,  2008.


\bibitem{xxlphd1}
P.~Balazs.
\newblock  Regular and Irregular Gabor Multipliers with Application to Psychoacoustic asking.
\newblock PhD thesis, University of Vienna, 2005.



\bibitem{xxlhar18}
P.~Balazs and H.~Harbrecht.
\newblock  Frames for the solution of operator equations in {H}ilbert spaces with fixed dual pairing.
\newblock {\em Numer. Funct. Anal. Optim.}, 
accepted, 2018.



 \bibitem{framepsycho16}
P.~Balazs, N.~Holighaus, T.~Necciari, and D.~Stoeva.
\newblock  Frame theory for signal prcoessing in psychoacoustics.
\newblock {\em  Excursions in Harmonic Analysis Vol. 5}, Springer,
 225--268, 2017.

 
\bibitem{kreizxxl1}
P.~Balazs, W.~Kreuzer and H.~Waubke.
\newblock  A stochastic 2{D}-model for
  calculating vibrations in liquids and soils.
\newblock {\em  J. Comput. Acoust.}, 15(3):271--283, 2007.




 \bibitem{xxlriek11}
P.~Balazs and G.~Rieckh.
\newblock  Oversampling operators: Frame representation of operators.
\newblock {\em Analele Universitatii "Eftimie Murgu"}, 
18(2):107--114, 2011.



 \bibitem{xxlrieck11}
P.~Balazs and G.~Rieckh.
\newblock  Redundant representation of operators.
\newblock arXiv:1612.06130.



 \bibitem{xxlstoeant11}
P.~ Balazs, D.~Stoeva and J.-P.~Antoine.
\newblock  Classification of General Sequences by Frame-Related Operators.
\newblock {\em Sampl. Theory Signal Image Process.}, 
10(2):151--170, 2011.


  
 \bibitem{bikhzh15}
H.~Bingyang, L.~Khoi, and K.~Zhu.
\newblock  Frames and operators in Schatten
  classes.
\newblock {\em Houston J. Math.}, 
41:1191--1219, 2015.

  
 \bibitem{brennscott1}
S.~Brenner and L.~Scott.
\newblock  The mathematical theory of finite element
  methods.
\newblock 2nd ed., Springer New York, 2002.


 \bibitem{Casaz1}
P.~Casazza.
\newblock  The Art of Frame Theory.
\newblock {\em Taiwanese J. Math.}, 
4(2):129--202, 2000. 


 
   \bibitem{caku13}
P.~{C}asazza and G.~{K}utyniok.
\newblock  {F}inite {F}rames {T}heory
  {A}nd {A}pplications.
\newblock {\em Appl. Numer. Harmon. Anal.}, 
{B}oston, {M}{A}: {B}irkh{\"a}user. xvi, 2013.
  

  


\bibitem{ole1n}
O.~Christensen.
\newblock  An Introduction to Frames and {R}iesz Bases, 2nd ed.
\newblock Birkh\"{a}user, Boston, 2016.


\bibitem{Chr08}
O.~Christensen.
\newblock  Frames and Bases: An Introductory Course.
\newblock Birkh\"{a}user, Boston, 2008.




\bibitem{olepinv}
O.~Christensen.
\newblock  Frames and pseudo-inverses.
\newblock {\em J. Math. Anal. Appl.}, 
195(2):401--414, 1995.



\bibitem{appchr}
O.~Christensen and R.~S.~Laugesen.
\newblock Approximately dual frames in Hilbert spaces and applications to Gabor frames.
\newblock{\em Sampl. Theory Signal Image Process.}, 9(3): 77-89, 2010.



\bibitem{enflom01}
P~ Enflo and V.~Lomonosov.
\newblock Some aspects of the invariant subspace
  problem.
\newblock Amsterdam: North-Holland,  533--559, 2001.



\bibitem{gohberg1}
I.~Gohberg, S.~Goldberg and M.~A. Kaashoek.
\newblock{ Basic Classes of Linear Operators}.
\newblock Birkh{\"a}user, 2003.



\bibitem{gr01}
K.~Gr{\"o}chenig.
 \newblock{Foundations of time-frequency analysis},
  Birkh{\"a}user Boston, 2001.

  
  
\bibitem{HarSch06}
H.~Harbrecht and R.~Schneider.
\newblock Wavelet Galerkin schemes for boundary integral equation. 
Implementation and quadrature.
\newblock{\em SIAM J. Sci. Comput.}, 27:1347--1370, 2006.



\bibitem{harbr08}
H.~Harbrecht, R.~Schneider, and C.~Schwab.
\newblock Multilevel frames for sparse tensor product spaces.
\newblock {\em Numer. Math.}, 110(2):199--220, 2008.



\bibitem{Kreuzeretal09}
W.~Kreuzer, P.~Majdak and Z.~Chen,
 \newblock Fast multipole boundary element
  method to calculate head-related transfer functions for a wide frequency
  range.
 \newblock {\em   J. Acoust .Soc. Am.}, 126(3):1280--1290, 2009.


\bibitem{liog04}
S.~Li and H.~Ogawa.
\newblock Pseudoframes for subspaces with applications.
\newblock{\em J. Fourier Anal. Appl.}, 10:409--431, 2004.


\bibitem{Pi}
A.~Pietsch.
\newblock{ Operator Ideals.}
\newblock North-Holland publishing Company, 1980.



\bibitem{sauter2010boundary}
S.~Sauter and C.~Schwab.
\newblock{ Boundary Element Methods.}
\newblock Springer Berlin Heidelberg, 2010.


\bibitem{Schatten}
R.~Schatten.
\newblock{ Norm Ideals of Completely Continuous Operators.}
\newblock Springer Berlin, 1960.


\bibitem{mitrak}
M.~Shamsabadi and A.~Arefijamaal.
\newblock Some results of $K$-frames and their multipliers.
\newblock arXiv:1807.08200.


\bibitem{mitra}
M.~Shamsabadi and A.~Arefijamaal.
\newblock The invertibility of fusion frame multipliers.
\newblock {\em Linear Multilinear Algebra}, 65(5):1062--1072, 2016.





\bibitem{spexxl14}
M.~Speckbacher and P.~Balazs.
\newblock Reproducing pairs and the continuous nonstationary {G}abor transform on LCA roups.
\newblock {\em J. Phys. A}, 48:395--201, 2015.


\bibitem{Stevenson03}
R.~ Stevenson. 
\newblock{Adaptive solution of operator equations using wavelet
  frames.} 
\newblock {\em   SIAM J. Numer. Anal. }, 41(3):1074--1100, 2003.





\bibitem{uncconv2011}

D.~T.~Stoeva and P.~Balazs.
\newblock Canonical forms of unconditionally convergent multipliers.
\newblock{\em J. Math. Anal. Appl.}, 399:252-259, 2013. 


\bibitem{bstable09}
D.~T. Stoeva and P.~Balazs.
\newblock{Detailed characterization of conditions for
  the unconditional convergence and invertibility of multipliers.} 
\newblock {\em   Sampl. Theory Signal Image Process.  }, 12(2-3):87--126, 2013.



\bibitem{balsto09new}
D.~T. Stoeva and P.~Balazs.
\newblock Invertibility of multipliers.
\newblock {\em Appl. Comput. Harmon. Anal.}, 33(2):292--299, 2012.



\bibitem{weidm80}
J.~Weidmann.
\newblock{ Linear Operators in Hilbert Spaces.}
\newblock Springer New York, 1980.


\bibitem{Har}
R.~Young.
\newblock{ An Introduction to Nonharmonic Fourier Series.}
\newblock Academic Press, New York, 1980 (revised first edition 2001).


\bibitem{zhuo}
K.~Zhu.
\newblock Operator Theory in Function Spaces.
\newblock Marcel Dekker, Inc, 1990.

\end{thebibliography}
\end{document}